\documentclass[12pt]{article}

\usepackage{amsmath,amssymb,amsthm}

\usepackage{enumerate}

\usepackage{mathrsfs}

\usepackage[all]{xy}
\usepackage{amscd}

\usepackage{enumitem}

\usepackage{geometry}
\usepackage{layout}
\newtheoremstyle{mystyle} 
    {\topsep}				
    {\topsep}				
    {\normalfont}			
    {} 					
    {\bfseries} 			
    {\newline}				
    {5pt plus 1pt minus 1pt}
    {\underline{\thmname{#1}\thmnumber{#2}\thmnote{（#3）}}}	

\theoremstyle{definition}
\newtheorem{theorem}{Theorem}
\newtheorem{prop}[theorem]{Proposition}
\newtheorem{lem}[theorem]{Lemma}
\newtheorem{cor}[theorem]{Corollary}

\newtheorem{definition}[theorem]{Definition}
\newtheorem{rem}[theorem]{Remark}

\numberwithin{theorem}{section} 	 
\numberwithin{equation}{section}		 

\newcommand \defeq{\overset{\text{def}}=}
\newcommand \ab{\operatorname{ab}}
\newcommand \Gal{\operatorname{Gal}}
\newcommand \isom {\overset \sim \rightarrow}

\newcommand \Ker{\operatorname{Ker}}

\def \Image{\operatorname{Im}}

\def \Aut{\operatorname{Aut}}
\def \tor{\operatorname{tor}}
\def \inf{\operatorname{inf}}
\def \Ind{\operatorname{Ind}}
\def\p{\frak{p}}
\def\q{\frak{q}}

\def\ur{\operatorname{ur}}

\def\prim{\operatorname{prim}}
\def\Ann{\operatorname{Ann}}
\def\Frob{\operatorname{Frob}}

\def\Hom{\operatorname{Hom}}
\def\rank{\operatorname{rank}}
\def\rk{\operatorname{rk}}
\def\dim{\operatorname{dim}}

\def\Dec{\operatorname{Dec}}
\def\fp{{\frak p}}

\def\fm{{\frak m}}
\def\fd{{\frak d}}

\newcommand\bZ{\mathbb Z}

\newcommand\bQ{\mathbb Q}
\newcommand\bR{\mathbb R}
\newcommand\bC{\mathbb C}
\newcommand\bF{\mathbb F}

\newcommand\ltilde{\tilde{l}}
\newcommand\chil{\chi^{(l)}}
\newcommand\chilp{\chi_{\p}^{(l)}}
\newcommand\bap{\overline{\p}}
\newcommand\baq{\overline{\q}}
\def \cs{\operatorname{cs}}
\def \Ram{\operatorname{Ram}}
\def\tr{\operatorname{tr}}
\def\sup{\operatorname{sup}}
\def\inf{\operatorname{inf}}

\def \s{\operatorname{s}}
\def \ff{\operatorname{ff}}

\def \naive{\operatorname{naive}}

\begin{document}

\renewcommand{\thesection}{\arabic{section}}

\renewcommand\thefootnote{*\arabic{footnote}}

\title{The Neukirch-Uchida theorem with restricted ramification}
\author{Ryoji Shimizu}
\date{
}
\maketitle


\begin{abstract}
Let $K$ be a number field and $S$ a set of primes of $K$.
We write $K_S/K$ for the maximal extension of $K$ 
unramified outside $S$ and $G_{K,S}$ for its Galois group. 
In this paper, we prove the following generalization of the Neukirch-Uchida theorem under some assumptions: 
``For $i=1,2$, let $K_i$ be a number field and $S_i$ a set of primes of $K_i$. If $G_{K_1,S_1}$ and $G_{K_2,S_2}$ are isomorphic, then $K_1$ and $K_2$ are isomorphic.''
Here the main assumption is that 
the Dirichlet density of $S_i$ is not zero for 
at least one $i$.
A key step of the proof is to recover group-theoretically the $l$-adic cyclotomic character of an open subgroup of $G_{K,S}$ for some prime number $l$.
\end{abstract}

\tableofcontents

\section{Introduction}
Let $K$ be a number field and $S$ a set of primes of $K$.
We write $K_S/K$ for the maximal extension of $K$
 unramified outside $S$ and $G_{K,S}$ for its Galois group. 
The goal of this paper 
is to prove the following generalization of the Neukirch-Uchida theorem\footnote{
For the original version of the Neukirch-Uchida theorem, see \cite{Neukirch}
and \cite{Uchida}.
} 
under as few assumptions as possible: 
``For $i=1,2$, let $K_i$ be a number field and $S_i$ a set of primes of $K_i$. If $G_{K_1,S_1}$ and $G_{K_2,S_2}$ are isomorphic, then $K_1$ and $K_2$ are isomorphic.''
For this, as in the proof of the Neukirch-Uchida theorem (cf. \cite{NSW}, Chapter X\hspace{-.1em}I\hspace{-.1em}I), we first characterize group-theoretically the decomposition groups in $G_{K,S}$,
 and then obtain an isomorphism of fields using them.

In the previous work \cite{Ivanov}, Ivanov showed that if $S$ contains all primes at infinity and all nonarchimedean primes above some prime $l$, the data of the $l$-adic cyclotomic character of some open subgroup of $G_{K,S}$ is equivalent to the data of the decomposition groups in $G_{K,S}$ at nonarchimedean primes in $S$.
Motivated by this, 
in \S 1, assuming that $S$ contains all primes at infinity and all nonarchimedean primes above some prime $l$ and that $S$ satisfies a certain condition $(\star_l)$ (Definition \ref{1.15}), 
we recover group-theoretically the $l$-adic cyclotomic character of $G_{K,S}$ modulo the torsion part (Theorem \ref{1.17}).
In particular, taking an open subgroup $U$ of $G_{K,S}$ corresponding to an extension of $K$ containing $\mu_l$ (resp. $\mu_4$) if $l \neq 2$ (resp. $l=2$), we recover the $l$-adic cyclotomic character of $U$.
The proof is based on the previous work \cite{Saidi-Tamagawa}, $\S 4$, where 
it plays an important role to study carefully
 the structure of annihilators of certain modules
 over the (multivariable in general) Iwasawa algebra associated to the maximal pro-$l$ abelian torsion free quotient of $\Gal(\overline{K}/K)$.
Instead, in this paper, 
we consider all quotients of $G_{K,S}$ which are isomorphic to $\bZ_l$.
This makes it necessary to characterize group-theoretically the cyclotomic $\bZ_l$-extension, but allows the assumption on $S$ to be condition $(\star_l)$ which is weaker than ``the Dirichlet density of $S$ is not zero'' (Proposition \ref{1.20}).

In \S 2, based on the results in \S 1 and \cite{Ivanov}, we obtain the ``local correspondence'': 
a one-to-one correspondence between the sets of the decomposition groups in $G_{K_i,S_i}$ at nonarchimedean primes in $S_i$ for $i=1,2$ (Theorem \ref{2.5}).
Further, if the decomposition groups are large enough, this correspondence 
turns out to be ``good'' in the sense that 
some local invariants (for example, the residue characteristic, the order of the residue field and the set of Frobenius lifts)
are preserved.
It follows from \cite{Chenevier-Clozel} that the decomposition groups are actually large enough, under a certain extra assumption on $S_i$ for one $i$.

In \S 3, we 
develop a way to
show the existence of an isomorphism between $K_1$ and $K_2$ 
assuming that the good local correspondence holds.
In the previous work \cite{Ivanov3}, Ivanov showed that $K_1$ and $K_2$ are isomorphic, 
if 
$K_1$ is totally imaginary and Galois over $\bQ$, 
$S_1$ 
is ``stable''\footnote{
We say that 
$S_1$ is stable if 
there are a subset $S_0 \subset S_1$, a finite subextension $K_{1,S_1}/L_0/K_1$ and an $\epsilon \in \bR_{>0}$ such that
the set $S_0(L)$ of primes of $L$ obtained as the inverse image of $S_0$
 has Dirichlet density  $\delta(S_0(L))>\epsilon$ 
for any finite subextension $K_{1,S_1}/L/L_0$.
}, 
and so on.
The result in 
\S 3 is a generalization of this result.
We do not have to assume $K_i/\bQ$ is Galois, and 
our assumption about the 
density 
of $S_i$ is much weaker (see condition (b) in Proposition \ref{4f.3}).
In the argument 
some properties of 
the quotient of $G_{K,S}$ corresponding to 
the maximal multiple $\bZ_l$-extension of $K$, whose structure is essentially determined by the number of complex primes, 
plays an important role. 
However, we can also use the result
in the case that $K_i$ is totally real.

In order to use the result in 
\S 3, 
we 
prove in \S 4 some formulas 
about the Dirichlet density. 
By one of them, we obtain another way to 
show easily that $K_1$ and $K_2$ are isomorphic
if 
$K_i/\bQ$ is Galois for $i=1,2$, 
the good local correspondence holds, 
the Dirichlet density of $S_i$ is larger than $1/2$ for one $i$, 
and so on (Proposition \ref{3.3}).

In \S 5, we prove the main theorems
using the results obtained so far (Theorem \ref{4.0}, Theorem \ref{4.1} and Theorem \ref{4.3}).
Lastly we show that if the Dirichlet densities of $S_1$ and $S_2$ are large enough, 
$K_1$ and $K_2$ are isomorphic (Theorem \ref{4.9}). 

\section*{Acknowledgements}
The author would like to thank Professor Akio Tamagawa for 
helpful advices and carefully reading preliminary versions of the present paper.

The author 
would also like to thank 
the referee for 
careful comments on a former version of this paper, 
in particular, 
on the proof of Proposition \ref{2.0}.

\section*{Notations}
\begin{itemize}[leftmargin=*]
\item[$\bullet$]
Given a set $A$ we write $\# A$ for its cardinality.

\item[$\bullet$]
For a profinite group $G$, let $\overline{[G, G]}$ be the closed subgroup of $G$ which is (topologically) generated by the commutators in $G$. We write $G^{\ab} \defeq G/\overline{[G,G]}$
for the maximal abelian quotient of $G$.

\item[$\bullet$]
Given a profinite group $G$ 
and a prime number $l$, we write $G^{(l)}$ for the maximal
pro-$l$ quotient of $G$.

\item
Given a Galois extension $L/K$, we write $G(L/K)$ for its Galois group $\Gal(L/K)$.
Given a field $K$, we write $\overline{K}$ for a separable closure of $K$, 
and $G_K$ for the absolute Galois group $G(\overline{K}/K)$ of $K$.

\item
Given a field $K$, we write $K^{\ab}$ for the maximal abelian subextension of $\overline{K}/K$, which corresponds to the quotient $G_K \twoheadrightarrow G_K^{\ab}$.

\item
Given a field $K$ and a prime number $l$, we write $K^{(l)}$ for the maximal pro-$l$ subextension of $\overline{K}/K$, which corresponds to the quotient $G_K \twoheadrightarrow G_K^{(l)}$.

\item
A number field is a finite extension 
 of the field of rational numbers $\bQ$. 
For an (a possibly infinite) algebraic extension $F$ of $\bQ$, 
we write 
$\widetilde{F}$ for the Galois closure of $F/\bQ$, 
$P=P_F$ 
for the set of primes  
of $F$, 
$P_\infty=P_{F,\infty}$ for the set of archimedean primes of $F$, 
and, 
for a prime number $l$, 
$P_l=P_{F,l}$ for the set of nonarchimedean primes of $F$ above $l$. 
Further, for a set of primes $S \subset P_F$, 
we set $S_f\defeq S \setminus P_\infty$, 
$P_S \defeq \{ p \in P_{\bQ} \mid P_{F,p} \subset S \}$.
For $\bQ \subset F \subset F' \subset \overline{\bQ}$, 
we write 
$S(F')$ for the set of primes of $F'$ above the primes in $S$: 
$S(F') \defeq \{ \p \in P_{F'} \mid \p|_F \in S \}$.
For convenience, we consider 
that $F'/F$ is ramified at a complex prime of $F'$ if it is above a real prime of $F$.
We write $F_S/F$ for the maximal extension of $F$ unramified outside $S$ and $G_{F,S}$ for its Galois group.
When $P_\infty \subset S$, we set 
$\mathcal O_{F,S} \defeq \{ a \in F \mid v_\p(a) \geq 0 \text{ for all } \p \notin S \}$, where $v_\p$ is the (normalized) exponential 
valuation associated to $\p$.

\item
Given an algebraic extension $K$ of $\bQ$ and $\p \in P_{K,f}$, 
we write $\kappa(\p)$ for the residue field at $\p$. 
When $K$ is a number field, we write $K_\p$ for the completion of $K$ at $\p$, 
and, in general, we write $K_\p$ for the union of $K'_{\p|_{K'}}$ for finite subextensions $K'/\bQ$ of $K/\bQ$.

\item
Let $L/K$ be a finite extension of number fields and $\q \in P_{L,f}$, 
and set $\p = \q|_{K}$.
We write $f_{\q,L/K} \defeq [\kappa(\q):\kappa(\p)]$.
We also write $f_{\p,L/K} = f_{\q,L/K}$, when no confusion arises.
We write $\cs(L/K)$ 
(resp. $\Ram(L/K)$) 
for the set of nonarchimedean primes of $K$
which split completely 
(resp. are ramified) 
in $L/K$.

\item
Let $K$ be a number field and $\p \in P_{K,f}$, 
and set $p = \p|_{\bQ}$.
Define the residual degree (resp. the local degree) of $\p$ 
as 
 $f_{\p,K/\bQ}$ 
(resp. $[K_\p:\bQ_p]$).
We set $\frak{N}(\p) \defeq p^{f_{\p,K/\bQ}}$.

\item
For a number field $K$ and a set of primes $S \subset P_K$, 
we set 
\begin{equation*}
\delta_{\sup}(S) \defeq  
\limsup_{s \to 1+0} \frac{\sum_{\p \in S_f} \frak{N}(\p)^{-s}}{\log{\frac{1}{s-1}}}
,\ 
\delta_{\inf}(S) \defeq  
\liminf_{s \to 1+0} \frac{\sum_{\p \in S_f} \frak{N}(\p)^{-s}}{\log{\frac{1}{s-1}}}
\end{equation*}
and 
if $\delta_{\sup}(S) = \delta_{\inf}(S)$, 
then write $\delta(S)$ (the Dirichlet density of $S$) for them.
The term 
``$\delta(S) \neq 0$''
will always mean that 
$S$ has positive Dirichlet density 
or 
$S$ does not have Dirichlet density.
Note that 
$\delta(S) \neq 0$ 
if and only if 
$\delta_{\sup}(S) > 0$.

\item
For $\bQ \subset F \subset F' \subset 
\overline{\bQ}$ with $F'/F$ Galois, $\q \in P_{F',f}$ and $\p = \q|_F$, 
write $G(F'/F)_{\q}=
D_{\q}(F'/F) 
\subset G(F'/F)$ for the decomposition group (i.e. the stabilizer) of $\q$ in $G(F'/F)$. 
We sometimes write $D_{\q} = D_{\q}(F'/F)$, when no confusion arises.
Further, we also write 
$D_{\p} 
= G(F'/F)_{\p} 
= D_{\p}(F'/F) 
= D_{\q}(F'/F)$, when no confusion arises.
Note that $D_{\p}$ is only defined up to conjugation.
There exists a canonical isomorphism $D_{\q}(F'/F) \simeq G(F'_\q/F_\p)$,
and 
we will identify $D_{\q}(F'/F)$ with $G(F'_\q/F_\p)$ via this isomorphism.

\item
Let $p$ be a prime number.
A 
$p$-adic field is a finite extension 
 of the field of $p$-adic numbers $\bQ_p$. 
Let $\kappa$ be an (a possibly infinite) algebraic extension of $\bQ_p$.
We write $V_{\kappa}$ (resp. $I_{\kappa}$) the wild inertia (resp. inertia) subgroup of $G_\kappa$, 
and 
$\kappa^{\tr}$ (resp. $\kappa^{\ur}$) for the subextension of $\overline{\kappa}/\kappa$ corresponding to $V_{\kappa}$ (resp. $I_{\kappa}$), 
and set 
$G_\kappa^{\tr} \defeq G_\kappa/V_{\kappa}$ and 
$G_\kappa^{\ur} \defeq G_\kappa/I_{\kappa}$.
Let $\lambda/\kappa$ be a Galois extension.
We say that $G(\lambda/\kappa)$ is full if $\lambda$ is algebraically closed. 
We write $I(\lambda/\kappa)$ for the inertia subgroup of $G(\lambda/\kappa)$.
When $\kappa$ is a $p$-adic field, 
we say 
that 
an element of $G(\lambda/\kappa)/I(\lambda/\kappa)$ is the Frobenius element 
if it induces the automorphism of the residue field of $\lambda$ mapping every element to its $q$-th power, where $q$ is the order of the residue field of $\kappa$, 
and 
that an element of $G(\lambda/\kappa)$ is a Frobenius lift if its image under $G(\lambda/\kappa) \twoheadrightarrow 
G(\lambda/\kappa)/I(\lambda/\kappa)$ is equal to the Frobenius element.

\item
Given an abelian group $A$, 
we write $A_{\tor}$ for the torsion subgroup of $A$.

\item
Given an abelian profinite group $A$,
we write $\overline{A_{\tor}}$ for the closure in $A$ of $A_{\tor}$,
 and set $A^{/\tor} \defeq A/\overline{A_{\tor}}$.

\item
Let $G$ be a group which acts on an abelian group $A$.
We write $I_G A$ for the subgroup of $A$ generated by the elements $\sigma a-a,\ a\in A,\ \sigma \in G$, and set $A_G\defeq A/I_G A$.

\item
Given a commutative ring $R$, an $R$-module $M$,
 and a subset $S$ of $M$, we write $\langle S\rangle_R \subset M$
  (or simply $\langle S\rangle$ if there is no risk of confusion) for the $R$-submodule of $M$ generated by the elements of $S$. 
In particular, given $x \in M$ we set $\langle x \rangle = \langle x \rangle_R \defeq \langle \{ x \} \rangle_R \subset M$.
Given $x \in M$ we write $\Ann_R(x) \defeq \{ r \in R \mid rx = 0 \}$ for the annihilator of $x$ in $R$. 
We write $M_{R\text{-}\tor} \defeq \{ m \in M \mid rm = 0 \text{ for some non-zero-divisor }  r \in R\}$. 
Given $r \in R$ we write $M[r] \defeq \{ m \in M \mid rm = 0 \}$ 
and
$M[r^\infty] \defeq \bigcup_{n \in \bZ_{>0}} M[r^n]$.
An $R$-submodule $N$ of $M$ is called $R$-cofinite if the quotient $M/N$ is a finitely generated $R$-module.

\item
Given a field $K$, we write $\mu (K)$ for the group consisting of the roots of unity in $K$.
For $n \in \bZ_{>0}$ 
distinct from the characteristic of $K$, 
we write $\mu_n=\mu_n (\overline{K}) \subset \mu (\overline{K})$ for the subgroup of order $n$.
For a prime number $l$ distinct from the characteristic of $K$, 
we set 
$\mu_{l^\infty} \defeq \bigcup_{n \in \bZ_{>0}} \mu_{l^n}(\overline{K}) \subset \mu (\overline{K})$.

\item
Let $l$ be a prime number. We set $\ltilde \defeq l$\ (resp.\ $\ltilde \defeq 4$) for $l\neq2$\ (resp.\ $l=2$).

\end{itemize}

\section{Recovering the $l$-adic cyclotomic character}

In this section, 
let $K$ be a number field, and
fix a prime number $l$.
We set $\Sigma =\Sigma_K \defeq \{l,\infty\}(K)$.

Let $K_{\infty}/K$ be a $\bZ_l$-extension, and set $\Gamma \defeq G(K_{\infty}/K)$.
We write $K_{\infty,0}/K$ for the cyclotomic $\bZ_l$-extension, and set $\Gamma_0=\Gamma_{K,0} \defeq G(K_{\infty,0}/K)$.
Note that $K_{\infty}/K$ is unramified outside $\Sigma$.
For $\p\in P_{K}\setminus \Sigma$, we write $\Gamma_{\p}$ for the decomposition group at $\p$ in $\Gamma$.
Then $\Gamma_{\p}$ is topologically generated by the Frobenius element $\gamma_{\p}$ at $\p$.

Let $S \subset P_K$ be a set of primes of $K$.
In the rest of this section, we assume $S\supset \Sigma$.
Then $\mu_{l^\infty} \subset K_S$, and we write $\chil=\chi_K^{(l)}$ for the $l$-adic cyclotomic character $G_{K,S}\rightarrow \Aut(\mu_{l^\infty})={\bZ_l}^{\ast}$.
The composite of 
$\chil$ 
and the first projection of the decomposition ${\bZ_l}^{\ast}=({1+\tilde{l}\bZ_l})\times{({\bZ_l}^{\ast})_{tor}}$ factors as $G_{K,S}\twoheadrightarrow \Gamma_0\rightarrow 1+\ltilde\bZ_l$, 
where we write $w=w_K:\Gamma_0\rightarrow 1+\ltilde\bZ_l$ for the second morphism.

In this section, we recover group-theoretically $w$ from $G_{K,S}$ under a certain assumption.

We denote by $\Lambda^{\Gamma}\defeq \bZ_l [[\Gamma]]$ the associated complete group ring
 (cf. \cite{NSW}, Chapter V, \S 2).
We also write $\Lambda$ for $\Lambda^{\Gamma}$ if there is no risk of confusion. (The same applies hereinafter.)
Given a generator $\gamma$ of $\Gamma$, 
we have an isomorphism 
of $\bZ_l$-algebras
$\Lambda\isom \bZ_l[[T]]$,  $\gamma\mapsto 1+T$. 
(See \cite{NSW}, (5.3.5) Proposition.)
More generally, let $\mathcal O/\bZ_l$ be a finite extension of (complete) discrete valuation rings.
Then we denote by 
$\Lambda_\mathcal O=\Lambda_\mathcal O^{\Gamma}
\defeq \mathcal O[[\Gamma]]=\Lambda\otimes_{\bZ_l}\mathcal O
\simeq \mathcal O [[T]]$
 the associated complete group ring
 over $\mathcal O$ (cf. loc. cit.).
Note that this is a noetherian UFD.

Consider the exact sequence 
$1\to H_S\to G_{K,S}\to \Gamma \to 1$, 
where
$H_S\defeq \Ker (G_{K,S}\twoheadrightarrow \Gamma)$.
By pushing out this sequence by the projection 
$H_S\twoheadrightarrow X_S=X_S^{\Gamma}\defeq ({H_S}^{(l)})^{\ab}$
we obtain an exact sequence $1\to X_S\to Y_S\to \Gamma \to 1$.
Note that 
$G_{K,S}\twoheadrightarrow Y_S$ can be reconstructed group-theoretically from $G_{K,S}\twoheadrightarrow \Gamma$ by its very definition,
and
 $X_S$ has a natural structure of $\Lambda$-module of which $\Ker (X_S\twoheadrightarrow X_\Sigma)$ is a $\Lambda$-submodule.

We set
\begin{equation*}
\begin{split} 
&(S\setminus \Sigma)^{fd} \defeq \{ \p\in S\setminus \Sigma \mid \text{ $\p$ is finitely decomposed in $K_{\infty}/K$} \} ,\\
&(S\setminus \Sigma)^{cd} \defeq \{ \p\in S\setminus \Sigma \mid \text{ $\p$ is completely decomposed in $K_{\infty}/K$} \}.
\end{split}
\end{equation*}
Note that 
$S\setminus \Sigma = (S\setminus \Sigma)^{fd} \coprod (S\setminus \Sigma)^{cd}$.

For $\p\in (S\setminus \Sigma)^{fd}$ with $\mu_l\subset K_\p$,
the local $l$-adic cyclotomic character 
$G_{K_\p}\rightarrow \Aut(\mu_{l^\infty})={\bZ_l}^{\ast}$
factors as $G_{K_\p}\twoheadrightarrow \Gamma_{\p}\rightarrow 
{\bZ_l}^{\ast}$
because $\Gamma_{\p} = G(K_\p(\mu_{l^\infty})/K_\p)$,
where we write $\chilp:\Gamma_{\p} \rightarrow {\bZ_l}^{\ast}$ for the second homomorphism.
Further, when $\mu_{\ltilde}\subset K_\p$ and $\Gamma =\Gamma_0$, we have $w|_{\Gamma_{\p}}=\chilp$.

\begin{prop}\label{1.1}
Assume that the weak Leopoldt conjecture holds for $K_{\infty}/K$.
Then there exists a canonical exact sequence of $\Lambda$-modules
\begin{equation*}
0\to \underset{\p\in S\setminus \Sigma}{\prod} \Ind_{\Gamma}^{\Gamma_{\p}}(I(K_\p^{(l)}/K_\p)_{
G_{K_{\infty,\p_\infty}}
})\to X_S\to X_{\Sigma}\to 0,\end{equation*}
where 
$\p_\infty \in P_{K_\infty}$ is a prime above $\p$, 
$I(K_\p^{(l)}/K_\p)$ is considered as a $G_{K_{\p}}$-module in a canonical manner, 
and 
$\Ind$ stands for the compact induction 
(cf. \cite{NSW}, Notation above (11.3.5) Theorem).
Further,
\begin{equation*}
\Ind_{\Gamma}^{\Gamma_{\p}}(I(K_\p^{(l)}/K_\p)_{G_{K_{\infty,\p_\infty}}}) \simeq
\begin{cases}
\Lambda /\langle \gamma_{\p}-\chilp(\gamma_{\p})
\rangle, &\mu_l\subset K_{\p}\text{ and }\p\in (S\setminus \Sigma)^{fd}, \\
\Lambda /l^{t_\p}, &\mu_l\subset K_{\p}\text{ and }\p\in (S\setminus \Sigma)^{cd}, \\
0, &\mu_l\not\subset K_{\p}, 
\end{cases}
\end{equation*}
where $l^{t_\p}=\# \mu (K_\p)[l^\infty]$.
\end{prop}

\begin{proof}
If $S$ is a finite set, the first assertion follows from \cite{NSW}, (11.3.5) Theorem.
(In \cite{NSW}, Chapter X\hspace{-.1em}I, \S3, $K$ is assumed to be totally imaginary if $l=2$. However, in the proof of \cite{NSW}, (11.3.5) Theorem, this assumption is not used.)
If $S$ is not a finite set, passing to the projective limit over all finite subsets of $S$ we obtain the desired exact sequence.
The second assertion follows from the proof of \cite{NSW}, (11.3.5) Theorem.
\end{proof}

Given a finitely generated free $\bZ_l$-module $A$,
we set $A^{\prim}\defeq A\setminus lA$.
For any $\alpha \in A\setminus \{ 0\}$, 
there exist a unique element $\tilde \alpha \in A^{\prim}$ and a unique element $m_\alpha \in \bZ_{\geq 0}$, 
such that $l^{m_\alpha}\tilde \alpha =\alpha$.
In particular, 
for $\p\in (S\setminus \Sigma)^{fd}$
and $\gamma_{\p} \in \Gamma$, 
we always 
write 
$\tilde\gamma_{\p}$ (resp. $m_{\gamma_{\p}}$) 
for
the unique element of $\Gamma^{\prim}$ (resp. of $\bZ_{\geq 0}$) 
such that $\tilde\gamma_{\p}^{l^{m_{\gamma_{\p}}}} = \gamma_{\p}$,
and set $m_{\p}\defeq m_{\gamma_{\p}}$.
For $\p\in (S\setminus \Sigma)^{fd}$ with $\mu_{\ltilde} \subset K_{\p}$, 
we set $(\gamma_{\p}', \chilp(\gamma_{\p})')\defeq {(\gamma_{\p}, \chilp(\gamma_{\p}))} \tilde{} \in \Gamma \times (1+\ltilde\Bbb Z_l)$.

For $\p\in S\setminus \Sigma$, 
we set $J_{\p}=J_{\p}^{\Gamma}\defeq  \Ind_{\Gamma}^{\Gamma_{\p}}(I(K_\p^{(l)}/K_\p)_{G_{K_{\infty,\p_\infty}}})$ and 
$J=J^{\Gamma}\defeq \underset{\p\in S\setminus \Sigma}{\prod} J_{\p}$.
For $\p\in (S\setminus \Sigma)^{fd}$ with $\mu_{\ltilde} \subset K_{\p}$, 
we set $J_\p'\defeq J_\p[\gamma_{\p}'-\chilp(\gamma_{\p})'] \subset J_\p$.

\begin{lem}\label{1.2}
The weak Leopoldt conjecture is true for $K_{\infty}/K$
if and only if \\
$H^2(G(K_S/K_{\infty}),\ \bQ_l/\bZ_l)=0$.
Further, 
the weak Leopoldt conjecture is true for $K_{\infty,0}/K$.
\end{lem}
\begin{proof}
These assertions follow immediately from
\cite{NSW}, (2.6.9) Theorem, (10.3.22) Theorem and (10.3.25) Theorem.
\end{proof}

\begin{lem}\label{1.3}
Assume that the weak Leopoldt conjecture holds for $K_{\infty}/K$.
Then \\
$\# \{ \p \in (S\setminus \Sigma)^{cd} \mid \mu_l \subset K_{\p} \} <\infty$
if and only if 
$X_S[l^\infty]$ is a finitely generated $\Lambda$-module.
In particular, 
when $\mu_l\subset K$, 
these conditions are equivalent to $\# (S\setminus \Sigma)^{cd}<\infty$.
Further, 
$(S\setminus \Sigma)^{cd}=\emptyset$ for $K_{\infty,0}/K$.
\end{lem}

\begin{proof}
By Proposition \ref{1.1}, 
there exists a canonical exact sequence of $\Lambda$-modules
\begin{equation*}0\to J[l^{\infty}]\to X_S[l^{\infty}]\to X_{\Sigma}[l^{\infty}].\end{equation*}
Then, by \cite{NSW}, (11.3.1) Proposition, 
$X_{\Sigma}$ is a finitely generated $\Lambda$-module, 
hence $X_{\Sigma}[l^{\infty}]$ is also a finitely generated $\Lambda$-module.
(In \cite{NSW}, Chapter X\hspace{-.1em}I, \S3, $K$ is assumed to be totally imaginary if $l=2$. However, in the proof of \cite{NSW}, (11.3.1) Proposition, this assumption is not used.)

We set 
$T^{fd} \defeq \{ \p \in (S\setminus \Sigma)^{fd} \mid \mu_l \subset K_{\p} \}$, $T^{cd} \defeq \{ \p \in (S\setminus \Sigma)^{cd} \mid \mu_l \subset K_{\p} \}$.
For $\p\in T^{fd}$, 
by the isomorphism 
$\Lambda\isom \bZ_l[[T]]$,  $\tilde\gamma_{\p}\mapsto 1+T$, 
we have an isomorphism
\begin{equation*}J_{\p}
\simeq
\Lambda /\langle \gamma_{\p}-\chilp(\gamma_{\p})
\rangle \simeq \bZ_l[[T]]/\langle (1+T)^{l^{m_{\p}}}-\chilp(\gamma_{\p})\rangle.\end{equation*}
By \cite{NSW}, (5.3.1) Division Lemma, 
$J_{\p}$ is a free $\bZ_l$-module of rank $l^{m_{\p}}$, 
in particular, a torsion free abelian group.
In addition, 
for $\p\in S\setminus \Sigma$ with $\mu_l \not\subset K_{\p}$, 
$J_{\p}=0$ by Proposition \ref{1.1}.
Therefore, 
we obtain an exact sequence of $\Lambda$-modules
\begin{equation*}0\to \underset{\substack{\p\in T^{cd}}}{\prod} J_{\p}\to J\to \underset{\substack{\p\in T^{fd}}}{\prod} J_{\p}\to 0,\end{equation*}
and an exact sequence of $\Lambda$-modules
\begin{equation*}0\to (\underset{\substack{\p\in T^{cd}}}{\prod} J_{\p})[l^{\infty}]\to J[l^{\infty}]\to (\underset{\substack{\p\in T^{fd}}}{\prod} J_{\p})[l^{\infty}]=0.\end{equation*}
Thus, 
$X_S[l^\infty]$ is a finitely generated $\Lambda$-module
if and only if 
$(\underset{\substack{\p\in T^{cd}}}{\prod} J_{\p})[l^{\infty}]$ is a finitely generated $\Lambda$-module.
By Proposition \ref{1.1}, 
for $\p \in T^{cd}$, 
$J_\p \simeq \Lambda /l^{t_\p}$ ($\ t_\p \geq 1$) is a finitely generated $\Lambda$-module.
So if $\# T^{cd}<\infty$, 
$(\underset{\substack{\p\in T^{cd}}}{\prod} J_{\p})[l^{\infty}]$ is a finitely generated $\Lambda$-module.

If $\# T^{cd}=\infty$, 
$(\underset{\substack{\p\in T^{cd}}}{\prod} J_{\p})[l^{\infty}]
\supset \underset{\substack{\p\in T^{cd}}}{\bigoplus} J_{\p}
\twoheadrightarrow  \underset{\substack{\p\in T^{cd}}}{\bigoplus} \Lambda /l$
and 
$ \underset{\substack{\p\in T^{cd}}}{\bigoplus} \Lambda /l$ is not a finitely generated $\Lambda$-module, 
so 
$(\underset{\substack{\p\in T^{cd}}}{\prod} J_{\p})[l^{\infty}]$ is not a finitely generated $\Lambda$-module.

For $K_{\infty,0}/K$, 
$(S\setminus \Sigma)^{cd}=\emptyset$ because
$K_\p (\mu_{l^\infty})/K_\p$ is an infinite algebraic extension 
for $\p \in S\setminus \Sigma$.
\end{proof}

\begin{definition}\label{1.4}
We say that $K_{\infty}/K$ and $S$ satisfy condition $(\dagger)$
if 
the weak Leopoldt conjecture is true for $K_{\infty}/K$ 
and 
$\# (S\setminus \Sigma)^{cd}<\infty$.
We say that $K_{\infty}/K$ and $S$ satisfy condition $(\dagger')$
if 
$H^2(G(K_S/K_{\infty}),\bQ_l/\bZ_l)=0$ 
and 
$X_S[l^\infty]$ is a finitely generated $\Lambda$-module.
\end{definition}

Note that condition $(\dagger')$ depends only on $G_{K,S}\twoheadrightarrow \Gamma$.

\begin{lem}\label{1.4.5}
If 
$K_{\infty}/K$ and $S$ satisfy condition $(\dagger)$, 
then they satisfy condition $(\dagger')$.
Further, 
if $\mu_l\subset K$, the converse is true.
In particular, 
for any $\Sigma \subset S \subset P_K$, 
$K_{\infty,0}/K$ and $S$ satisfy conditions $(\dagger)$ and $(\dagger')$.
\end{lem}
\begin{proof}
The assertions follow immediately from Lemma \ref{1.2} and Lemma \ref{1.3}.
\end{proof}

\begin{lem}\label{1.5}
Let $\mathcal O/\bZ_l$ be a finite extension of (complete) discrete valuation rings and $\fm\subset\mathcal O$ the maximal ideal. Let $\gamma,\gamma' \in\Gamma$, and $\alpha,\alpha'\in 1+\fm$. 
If
there exists
$\nu\in\Bbb Z_l$ such that $\gamma=(\gamma')^\nu$ and $\alpha=(\alpha')^\nu$, 
then  
$\gamma'-\alpha'$ divides $\gamma-\alpha$ in $\Lambda_{\mathcal O}$ 
(i.e. $\gamma-\alpha\in\langle \gamma'-\alpha'\rangle_{\Lambda_{\mathcal O}}$).
Further, 
when $\gamma'\in \Gamma^{\prim}$, the converse is true.
In particular, $(\gamma'-\alpha') \vert (\gamma-\alpha)$ implies $\gamma\in \langle \gamma'\rangle$.
\end{lem}
\begin{proof}
We can apply the same proof as that of 
\cite{Saidi-Tamagawa}, Lemma 4.3, where they assume that $\gamma'\in \Gamma^{\prim}$.
\end{proof}

\begin{lem}\label{1.6}
Let
$(\gamma, \alpha), (\gamma', \alpha')\in  (\Gamma \times (1+\ltilde\Bbb Z_l))^{\prim}$.
Assume that 
$\gamma-\alpha$ and $\gamma'-\alpha'$ are not coprime, 
then 
there exists
$\nu\in\Bbb Z_l^{\ast}$ such that $\gamma=(\gamma')^\nu$ and $\alpha=(\alpha')^\nu$, 
in particular, 
$\langle \gamma-\alpha \rangle_{\Lambda}=\langle \gamma'-\alpha' \rangle_{\Lambda}$.
\end{lem}
\begin{proof}
When $\gamma' \in \Gamma^{\prim}$, 
$\gamma'-\alpha'$ is a prime element of $\Lambda$, hence 
$\gamma'-\alpha'$ divides $\gamma-\alpha$.
Therefore, by Lemma \ref{1.5}, 
there exists
$\nu\in\Bbb Z_l$ such that $\gamma=(\gamma')^\nu$ and $\alpha=(\alpha')^\nu$.
Since $(\gamma, \alpha)\in (\Gamma \times (1+\ltilde\Bbb Z_l))^{\prim}$, 
we have $\nu\in\Bbb Z_l^{\ast}$.
Hence $\gamma=(\gamma')^\nu \in \Gamma^{\prim}$, 
so $\gamma-\alpha$ is also a prime element of $\Lambda$.
Thus, 
$\langle \gamma-\alpha \rangle_{\Lambda}=\langle \gamma'-\alpha' \rangle_{\Lambda}$.
When $\gamma \in \Gamma^{\prim}$, the proof is the same.

When $\gamma,\gamma' \notin \Gamma^{\prim}$, 
$\alpha, \alpha' \in (1+\ltilde\Bbb Z_l)^{\prim}$.
We set $m \defeq m_\gamma$, 
and let 
$\mathcal O/\bZ_l$ be a finite extension of complete discrete valuation rings containing all $l^m$-th roots of $\alpha$ in $\overline{\bQ_l}$, 
and $E\defeq \mu_{l^m}\subset \mathcal O$, 
and 
write $\beta \in \mathcal O$ for an $l^m$-th root of $\alpha$.
Then 
$\Lambda_{\mathcal O}$ is a UFD, 
and
we have 
$\gamma-\alpha=
\tilde \gamma^{l^m}-\beta^{l^m}=
\prod_{\eta\in E}(\tilde \gamma-\eta \beta)$
in $\Lambda_{\mathcal O}$.
Note that $\tilde \gamma-\eta \beta$ is a prime element of $\Lambda_{\mathcal O}$ for $\eta\in E$.
Since 
$\gamma-\alpha$ and $\gamma'-\alpha'$ are not coprime in $\Lambda_{\mathcal O}$, 
$\tilde \gamma-\eta \beta$ is a prime divisor of $\gamma'-\alpha'$ 
for some $\eta\in E$.
By Lemma \ref{1.5}, 
there exists
$\nu' \in\Bbb Z_l$ such that $\gamma'=(\tilde \gamma)^{\nu'}$ and $\alpha'=(\eta \beta)^{\nu'}$.
Now, 
we set $\nu'=ul^k$ ($u \in \bZ_l^{\ast}, k \in \bZ_{\geq 0}$).
By $\alpha'^{u^{-1}}=(\eta \beta)^{l^k}$ and
$\alpha, \alpha'^{u^{-1}}\in (1+\ltilde\Bbb Z_l)^{\prim}$, 
we obtain $k=m$.
Indeed, 
if $k<m$, then $({\alpha'}^{u^{-1}})^{l^{m-k}}=(\eta \beta)^{l^m}=\alpha$, 
which contradicts $\alpha \in (1+\ltilde\Bbb Z_l)^{\prim}$, 
and 
if $k>m$, then ${\alpha'}^{u^{-1}}=((\eta \beta)^{l^m})^{l^{k-m}}=\alpha^{l^{k-m}}$, 
which contradicts $\alpha'^{u^{-1}} \in (1+\ltilde\Bbb Z_l)^{\prim}$.
Thus,
$\gamma'=\gamma^u$ and $\alpha'=\alpha^u$, 
so $\nu \defeq u^{-1}$ satisfies the desired property.
Further, by Lemma \ref{1.5}, 
we obtain $\langle \gamma-\alpha \rangle_{\Lambda}=\langle \gamma'-\alpha' \rangle_{\Lambda}$.
\end{proof}

\begin{rem}
When $l \neq 2$, 
the condition $(\gamma, \alpha) \in  (\Gamma \times (1+\ltilde\Bbb Z_l))^{\prim}$
implies that 
$\gamma-\alpha$ is a prime element of $\Lambda$, 
even if $\gamma \notin \Gamma^{\prim}$.
Indeed, let $\fm\subset\Lambda$ be the maximal ideal of $\Lambda$, 
then 
the image of $\gamma-\alpha \in \fm$ in $\fm/\fm^2$ is not $0$.
Thus, 
we may simplify 
the proof of Lemma \ref{1.6} in this case.
\end{rem}

\begin{definition}
Let $M$ be a $\Lambda$-module.
We define a set of characters of $\Gamma$ as follows: 
\begin{equation*}
A_M^\Gamma 
\defeq 
\left\{ \rho :\Gamma\rightarrow 1+\ltilde\bZ_l \left|
\begin{array}{l}
\text{For $(\gamma, \alpha)\in  (\Gamma \times (1+\ltilde\Bbb Z_l))^{\prim}$ 
and $x\in M\setminus \{0\}$}\\
\text{with $\gamma-\alpha \in \Ann_{\Lambda}(x)$, 
$\rho (\gamma)=\alpha$}
\end{array}
\right.\right\}.
\end{equation*}
Note that 
$A_M^\Gamma=A_{M_{\Lambda\text{-}\tor}}^\Gamma$, 
and
if $M\subset M'$ then $A_{M'}^\Gamma \subset A_M^\Gamma$.
\end{definition}

We first consider the case $\Gamma =\Gamma_0$.

\begin{prop}\label{1.9}
Assume that 
$\mu_{\ltilde}\subset K$, 
$\Gamma =\Gamma_0$ and 
$S\setminus \Sigma \neq \emptyset$.
Then 
$A_J^{\Gamma_0}=\{ w\}$.
\end{prop}
\begin{proof}
Let 
$(\gamma, \alpha)\in  (\Gamma_0 \times (1+\ltilde\Bbb Z_l))^{\prim}$ 
and 
$x=( x_\p)_\p \in J\setminus \{0\}$ 
with 
$\gamma-\alpha \in \Ann_{\Lambda}(x)$.
Then 
for some $\p \in S\setminus \Sigma$, 
$x_\p \neq 0$ 
and 
$\Ann_{\Lambda}(x) \subset \Ann_{\Lambda}(x_\p)$.
Note that 
$w|_{\Gamma_{\p}}=\chilp$ and 
$(\gamma_{\p}', \chilp(\gamma_{\p})') = (\tilde \gamma_{\p}, w(\tilde \gamma_{\p}))$.
By Proposition \ref{1.1}, 
$J_\p \simeq \Lambda /\langle \gamma_{\p}-w(\gamma_{\p}) \rangle$.
Now, we set 
$E_{\p}\defeq \mu_{l^{m_\p}} \subset \overline{\Bbb Q}_l$, 
$\mathcal O_{E_{\p}}\defeq \Bbb Z_l[E_{\p}]=\Bbb Z_l[\zeta]\subset \overline{\Bbb Q}_l$ 
and 
$\Lambda_{E_{\p}}\defeq\Lambda_{\mathcal O_{E_{\p}}}$, 
where $\zeta$ is a primitive $l^{m_{\p}}$-th root of unity in $\overline{\Bbb Q}_l$.
Note that 
$\Lambda_{E_{\p}}$ is a UFD, 
$\gamma_{\p}-w(\gamma_{\p})=
\tilde \gamma_{\p}^{l^{m_{\p}}}-w(\tilde\gamma_{\p})^{l^{m_{\p}}}=
\prod_{\eta\in E_{\p}}(\tilde \gamma_{\p}-\eta w(\tilde\gamma_{\p}))$ 
in $\Lambda_{E_{\p}}$, 
and 
for each $\eta\in E_{\p}$, $\tilde \gamma_{\p}-\eta w(\tilde\gamma_{\p})$ is a prime element of $\Lambda_{E_{\p}}$.
Then we have 
\begin{equation*}
J_\p \simeq 
\Lambda /\langle \gamma_{\p}-w(\gamma_{\p}) \rangle
\hookrightarrow 
\Lambda_{E_{\p}} /\langle \gamma_{\p}-w(\gamma_{\p}) \rangle_{\Lambda_{E_{\p}}}
\hookrightarrow 
\prod_{\eta\in E_{\p}} \Lambda_{E_{\p}} /
\langle \tilde\gamma_{\p}-\eta w(\tilde\gamma_{\p}) \rangle_{\Lambda_{E_{\p}}}.
\end{equation*}
Here, the first injection comes from the fact that
$J_\p$ is a free $\bZ_l$-module 
as we have seen in the proof of Lemma \ref{1.3}, 
while the second injection comes from the Chinese remainder theorem.
We write $(y_\eta)_\eta$ for the image of $x_\p$ under this injection.
Since $x_\p \neq 0$, 
$y_\eta \neq 0$ for some $\eta\in E_{\p}$.
For this $\eta\in E_{\p}$, 
we have 
\begin{equation*}
\gamma-\alpha \in \Ann_{\Lambda}(x) 
\subset \Ann_{\Lambda}(x_\p) 
\subset \Ann_{\Lambda_{E_{\p}}}(y_\eta)
=\langle \tilde\gamma_{\p}-\eta w(\tilde\gamma_{\p}) \rangle_{\Lambda_{E_{\p}}}.
\end{equation*}
So 
$\langle \gamma-\alpha \rangle_{\Lambda_{E_{\p}}} \subset
\langle \tilde \gamma_{\p}-\eta w(\tilde\gamma_{\p})\rangle_{\Lambda_{E_{\p}}}$. 
By Lemma \ref{1.5}, 
there exists
$\nu\in\Bbb Z_l$ such that 
$\gamma=\tilde\gamma_{\p}^{\nu}$ and $\alpha=(\eta w(\tilde\gamma_{\p}))^{\nu}=\eta^{\nu}w(\tilde\gamma_{\p})^{\nu}$.
Since 
$1+\ltilde\Bbb Z_l$ is torsion free, $\eta^{\nu}=1$.
Hence we have $\alpha=w(\tilde\gamma_{\p}^{\nu})=w(\gamma)$.
Therefore, we obtain 
$w \in A_J^{\Gamma_0}$.

Take $\rho \in A_J^{\Gamma_0}$.
For $\p \in S\setminus \Sigma$, 
$(\tilde\gamma_{\p}, w(\tilde\gamma_{\p}))\in  (\Gamma_0 \times (1+\ltilde\Bbb Z_l))^{\prim}$, 
$\gamma_{\p}=\tilde\gamma_{\p}^{l^{m_{\p}}}$ and 
$w(\gamma_{\p})=w(\tilde\gamma_{\p})^{l^{m_{\p}}}$.
So by Lemma \ref{1.5}, 
$\tilde\gamma_{\p}-w(\tilde\gamma_{\p})$ divides 
$\gamma_{\p}-w(\gamma_{\p})$.
Hence 
if 
we write $x_\p \in J_\p \subset J$ for
the image of the quotient 
$(\gamma_{\p}-w(\gamma_{\p}))/(\tilde\gamma_{\p}-w(\tilde\gamma_{\p}))$
under 
$\Lambda \twoheadrightarrow 
\Lambda /\langle \gamma_{\p}-w(\gamma_{\p}) \rangle
\simeq J_\p$, 
then 
$\tilde\gamma_{\p}-w(\tilde\gamma_{\p}) \in \Ann_{\Lambda}(x_\p)$.
Since $\tilde\gamma_{\p}-w(\tilde\gamma_{\p}) \notin \Lambda^{\times}$, $x_\p \not = 0$.
Therefore, we obtain 
$\rho(\tilde\gamma_{\p})=w(\tilde\gamma_{\p})$
by the definition of $A_J^{\Gamma_0}$.
Since $\tilde\gamma_{\p}$ is a generator of $\Gamma_0$, 
we have $\rho =w$.
\end{proof}

\begin{prop}\label{1.10}
Assume that 
$\mu_{\ltilde}\subset K$, 
$\Gamma =\Gamma_0$ and 
$\# (S\setminus \Sigma)= \infty$.
Let $M \subset J$ be a $\Lambda$-cofinite $\Lambda$-submodule.
Then 
$A_M^{\Gamma_0}=\{ w\}$.
In particular, 
for a $\Lambda$-cofinite $\Lambda$-submodule $M' \subset X_S$, 
$A_{M'}^{\Gamma_0}\subset A_{M'\cap J}^{\Gamma_0}=\{ w\}$.
\end{prop}
\begin{proof}
By Proposition \ref{1.9}, 
$A_M^{\Gamma_0}\supset A_J^{\Gamma_0}=\{ w\}$.
So we show the converse.
Take any $\rho \in A_M^{\Gamma_0}$ and a generator $\gamma$ of $\Gamma_0$.
For each $\p \in S\setminus \Sigma$, by Lemma \ref{1.5}, 
$\langle \tilde \gamma_{\p}-w(\tilde\gamma_{\p})\rangle_{\Lambda}=
\langle \gamma-w(\gamma)\rangle_{\Lambda}$.
Hence 
$J_\p'= J_\p[\gamma-w(\gamma)] \simeq \Lambda /\langle \gamma-w(\gamma) \rangle$.
Now, we set 
$J'\defeq \prod_{\p \in (S\setminus \Sigma)}J_\p'\subset J$, 
which is not a finitely generated $\Lambda$-module.
Since $J/M$ is a finitely generated $\Lambda$-module, 
we have 
$J'\cap M=\Ker(J'\subset J\to J/M)\neq \{0\}$.
So, take $x \in J'\cap M \setminus \{ 0 \}$.
Then
$\Ann_{\Lambda}(x)=
\langle \gamma-w(\gamma)\rangle_{\Lambda} \ni \gamma-w(\gamma)$.
Thus, by the definition $A_M^{\Gamma_0}$, 
we obtain $\rho(\gamma)=w(\gamma)$.
Since $\gamma$ is a generator of $\Gamma_0$, 
we have $\rho =w$.

The second assertion follows from 
the first assertion 
and 
the fact that 
for a $\Lambda$-cofinite $\Lambda$-submodule $M' \subset X_S$, 
$M' \cap J \subset J$ is $\Lambda$-cofinite.
\end{proof}

Next, we consider the case of a general $\Gamma$.

\begin{definition}
We say that $K_{\infty}/K$ and $S$ satisfy condition $(\ast)$
if 
for any $T\subset S\setminus \Sigma$ with $\# ((S\setminus \Sigma)\setminus T)<\infty$, 
the Frobenius elements at primes in $T$ generate an open subgroup of $G(K_{\infty}K_{\infty,0}/K)$.
\end{definition}

\begin{rem}\label{1.12}
If $K_{\infty}=K_{\infty,0}$, 
then $G(K_{\infty}K_{\infty,0}/K)=G(K_{\infty,0}/K)\simeq \bZ_l$.
Hence condition $(\ast)$ is equivalent to $\# S=\infty$.
If $K_{\infty} \neq K_{\infty,0}$, 
then 
$G(K_{\infty}K_{\infty,0}/K) \simeq \bZ_l^2$.
So, a closed subgroup of $G(K_{\infty}K_{\infty,0}/K)$ is an open subgroup 
if and only if 
it is a free $\bZ_l$-module of rank $2$.
\end{rem}

\begin{lem}\label{1.13}
Assume that 
$\mu_{\ltilde}\subset K$,  
$\Gamma \neq \Gamma_0$ and  
the Frobenius elements at $\p$ and $\q \in S\setminus \Sigma$ generate an open subgroup of $G(K_{\infty}K_{\infty,0}/K)$.
Then, there is no character $\rho :\Gamma\rightarrow 1+\ltilde\bZ_l$ 
such that 
$\rho (\gamma_{\p})=\chilp(\gamma_{\p})$ 
and 
$\rho (\gamma_{\q})=\chi_{\q}^{(l)}(\gamma_{\q})$.
\end{lem}
\begin{proof}
We set 
$G \defeq G(K_{\infty}K_{\infty,0}/K)$ 
and let 
$U$ be the open subgroup of $G$ generated by the Frobenius elements at $\p$ and $\q$.
By assumption, 
the pullbacks of 
$f_1 : \Gamma \times \Gamma_0 \overset{pr_1}\twoheadrightarrow \Gamma \overset{\rho}\to 1+\ltilde\bZ_l$
and 
$f_2 : \Gamma \times \Gamma_0 \overset{pr_2}\twoheadrightarrow \Gamma_0 \overset{w}\to 1+\ltilde\bZ_l$
by
$U\hookrightarrow G\hookrightarrow \Gamma \times \Gamma_0$
coincide. 
Since the cokernel of
$U\hookrightarrow G\hookrightarrow \Gamma \times \Gamma_0$ is finite 
and 
$1+\ltilde\bZ_l$ is a torsion free $\bZ_l$-module, 
$f_1$ and $f_2$ 
coincide.
Since $\Ker(pr_1)$ and $\Ker(pr_2)$ generate $\Gamma \times \Gamma_0$, 
$f_1$ and $f_2$ 
coincide with 
the trivial map 
$\Gamma \times \Gamma_0 \to 1 \to 1+\ltilde\bZ_l$.
But this contradicts the fact that 
$f_2$
is nontrivial.
\end{proof}

\begin{prop}\label{1.14}
Assume that 
$\mu_{\ltilde}\subset K$,  
$\Gamma \neq \Gamma_0$ and 
$K_{\infty}/K$ and $S$ satisfy conditions $(\dagger)$ and $(\ast)$.
Let $M \subset X_S$ be a $\Lambda$-cofinite $\Lambda$-submodule.
Then 
$A_M^{\Gamma}=\emptyset$.
\end{prop}
\begin{proof}
Replacing $M$ by $M\cap J$, we may assume that $M\subset J$ and $M$ is $\Lambda$-cofinite in $J$.
As $J/M$ is a finitely generated $\Lambda$-module, 
by \cite{NSW}, (5.3.8) Structure Theorem for Iwasawa Modules, 
there exist 
$r, s \in \bZ_{\geq 0}, n_j \in \bZ_{> 0}(1\leq j \leq s)$, 
prime elements $f_j(1\leq j \leq s)$  of $\Lambda$
and 
a $\Lambda$-homomorphism 
$J/M \to \Lambda^{\oplus r} \oplus 
\bigoplus_{j=1}^{s}\Lambda/f_j^{n_j}$
such that 
the kernel and the cokernel of this homomorphism are finite.

For $\p, \q \in (S\setminus \Sigma)^{fd}$, 
we say that $\p$ and $\q$ are equivalent 
if and only if 
\begin{equation*}\langle \gamma_{\p}'-\chilp(\gamma_{\p})' \rangle_{\Lambda}
=
\langle \gamma_{\q}'-\chilp(\gamma_{\q})' \rangle_{\Lambda}.\end{equation*}
This is an equivalence relation on $(S\setminus \Sigma)^{fd}$.
We write 
$(S\setminus \Sigma)^{fd}=\coprod_{i \in I}T_i$ 
for the union of all equivalence classes.
For each $i \in I$, take a $\p_i \in T_i$ 
and set 
$a_i\defeq \gamma_{\p_i}'-\chilp(\gamma_{\p_i})'$
and 
$J_{T_i}'\defeq \prod_{\p \in T_i}J_\p'\subset J$.
Then we have $J_{T_i}'\simeq \prod_{\p \in T_i}\Lambda /a_i$.
Now, the set 
\begin{equation*}I' \defeq \{ i\in I \mid \text{$\# (J/M)[a_i] =\infty$ and $\# T_i<\infty$}\}\end{equation*} is finite.
Indeed, 
since 
$a_i$ and $a_j$ are coprime in $\Lambda$ for $i \neq j$ by Lemma \ref{1.6}, 
the number of $i \in I$ with $\# (J/M)[a_i] =\infty$ is at most $s$
by the structure of $J/M$ mentioned above.
So, 
$\coprod_{i \in I'}T_i$ is a finite set.
Then, 
for $i\in I\setminus I'$, 
we have $J_{T_i}'\cap M=\Ker(J_{T_i}'\subset J[a_i]\to (J/M)[a_i])\neq \{0\}$.
Indeed, 
if $\# (J/M)[a_i] \neq\infty$, 
this follows from 
the fact $\# J_{T_i}' =\infty$.
If $\# (J/M)[a_i] =\infty$, by the definition of $I'$, 
$\# T_i=\infty$.
Hence, $J_{T_i}'$ is not a finitely generated $\Lambda$-module.
Thus, as $(J/M)[a_i]$ is a finitely generated $\Lambda$-module, 
we have 
\begin{equation*}J_{T_i}'\cap M=\Ker(J_{T_i}'\subset J[a_i]\to (J/M)[a_i])\neq \{0\}.\end{equation*}
Therefore, for $i\in I\setminus I'$, 
we can take $x_i \in J_{T_i}'\cap M \setminus \{ 0 \}$.
Then, for $\p \in T_i$, 
$\gamma_{\p}'-\chilp(\gamma_{\p})'\in \Ann_{\Lambda}(x_i)$.
If 
$\rho \in A_M^{\Gamma}$, 
we have $\rho(\gamma_{\p}')=\chilp(\gamma_{\p})'$, 
in particular, $\rho(\gamma_{\p})=\chilp(\gamma_{\p})$.
Now, by condition $(\dagger)$, $\coprod_{i \in I'}T_i$ and $(S\setminus \Sigma)^{cd}$ are finite sets.
Therefore, 
by condition $(\ast)$, 
the Frobenius elements at primes in $\coprod_{i \in I\setminus I'}T_i$ generate an open subgroup of $G(K_{\infty}K_{\infty,0}/K)$.
Thus, 
by Remark \ref{1.12} and Lemma \ref{1.13}, 
we obtain 
$A_M^{\Gamma}=\emptyset$.
\end{proof}

\begin{definition}\label{1.15}
We say that $S$ satisfies condition $(\star_l)$
if 
for any $\bZ_l$-extension $K_{\infty}/K$, 
$K_{\infty}/K$ and $S$ satisfy condition $(\ast)$.
\end{definition}

\begin{lem}\label{1.16}
Assume that $K_{\infty}/K$ and $S$ satisfy condition $(\ast)$.
Then for any finite subextension $K_S/L/K$, 
$K_{\infty}L/L$ and $S(L)$ satisfy condition $(\ast)$.
\end{lem}
\begin{proof}
If $\Gamma = \Gamma_0$, 
as we have seen in Remark \ref{1.12}, 
this lemma only states that 
if $\# S = \infty$, then $\# S(L) = \infty$, 
and this is obvious.
If $\Gamma \neq \Gamma_0$, then 
$G(K_{\infty}K_{\infty,0}/K)\simeq \bZ_l^2$ 
and 
a closed subgroup of $G(K_{\infty}K_{\infty,0}/K)$ is open 
if and only if 
it is a free $\bZ_l$-module of rank $2$, 
as we have seen in Remark \ref{1.12}.
If the Frobenius elements at $\p, \q \in S\setminus \Sigma$ generate an open subgroup of $G(K_{\infty}K_{\infty,0}/K)$, 
take $\p', \q' \in S(L)\setminus \Sigma_L$ above $\p, \q$, respectively, 
then the Frobenius elements at $\p', \q'$ generate an open subgroup of $G(K_{\infty}K_{\infty,0}L/L)$.
The assertion follows from this.
\end{proof}

With the notations and assumptions as in Lemma \ref{1.16}, 
note that 
\begin{equation*}G(K_{\infty}L/L)=G_{L,S(L)}/\Ker(G_{L,S(L)} \hookrightarrow G_{K,S} \twoheadrightarrow \Gamma).\end{equation*}

\begin{theorem}\label{1.17}
Assume that $S$ satisfies condition $(\star_l)$.
Then 
the surjection $G_{K,S}\twoheadrightarrow \Gamma_0$ and $w:\Gamma_0\rightarrow 1+\ltilde\bZ_l$ are characterized group-theoretically from $G_{K,S}$ (and $l$).
\end{theorem}
\begin{proof}
Let $K_\infty/K$ be a $\bZ_l$-extension 
satisfying the condition that 
$K_{\infty}L/L$ and $S(L)$ satisfy condition $(\dagger')$ 
for any finite subextension $K_S/L/K$.
By Lemma \ref{1.4.5}, 
$K_{\infty,0}/K$ satisfies this condition.
By the definition of condition $(\dagger')$, 
we can 
detect purely group-theoretically whether or not 
the $\bZ_l$-extension $K_\infty/K$ 
corresponding to 
a $\bZ_l$-quotient of $G_{K,S}$ 
satisfies the above condition. 
By condition $(\star_l)$, 
$K_\infty/K$ and $S$ satisfy condition $(\ast)$.

If $K_\infty = K_{\infty,0}$, 
there exists a finite subextension $K_S/L/K$ 
satisfying the following condition: 
``For any finite subextension $K_S/L'/L$, 
set $\Gamma' \defeq G(K_\infty L'/L')$, then 
there exists a $\Lambda$-cofinite $\Lambda^{\Gamma'}$-submodule $M_0 \subset X_{S(L')}^{\Gamma'}$ 
such that 
$A_M^{\Gamma'}\neq \emptyset$ 
for any $\Lambda^{\Gamma'}$-cofinite $\Lambda^{\Gamma'}$-submodule $M \subset M_0$''.
Indeed, 
we set $L=K(\mu_{\ltilde})$ and $M_0=J^{\Gamma'}$, 
then, 
by Proposition \ref{1.10}, 
$A_M^{\Gamma'}=\{ w_{L'}\} \neq \emptyset$ for any $M$ as above.

On the other hand, 
if $K_\infty \neq K_{\infty,0}$, 
there does not exist a finite subextension $K_S/L/K$ as above.
Indeed, 
for any finite subextension $K_S/L/K$, 
set $L' \defeq L(\mu_{\ltilde})$, 
then 
$K_\infty L'/L'$ and $S(L')$ satisfy 
condition $(\dagger)$ 
by Lemma \ref{1.4.5} 
and 
condition $(\ast)$ by Lemma \ref{1.16}.
Therefore, by Proposition \ref{1.14}, 
$A_M^{\Gamma'}=\emptyset$ for any $\Lambda^{\Gamma'}$-cofinite $\Lambda^{\Gamma'}$-submodule $M \subset X_{S(L')}^{\Gamma'}$.

Thus, 
we can 
detect group-theoretically whether or not 
the $\bZ_l$-extension $K_\infty/K$ 
corresponding to a $\bZ_l$-quotient of $G_{K,S}$
is the cyclotomic $\bZ_l$-extension. 


Further, 
if $K_\infty = K_{\infty,0}$, 
there exists a finite subextension $K_S/L/K$ 
satisfying the following condition: 
``For any finite subextension $K_S/L'/L$, 
there exists a $\Lambda^{\Gamma_{L',0}}$-cofinite $\Lambda^{\Gamma_{L',0}}$-submodule $M_0 \subset X_{S(L')}^{\Gamma_{L',0}}$ 
such that 
for any $\Lambda^{\Gamma_{L',0}}$-cofinite $\Lambda^{\Gamma_{L',0}}$-submodule $M \subset M_0$, 
$A_M^{\Gamma_{L',0}}$ consists of one element 
for which we write $w_{L'}'$, 
and the following diagram commutes: 
\begin{equation*}
\xymatrix{
\Gamma_{L',0}\ar@{^{(}-{>}}[rr] \ar[rd]_-{w_{L'}'}\ar@{}[d]&&\Gamma_{L,0}\ar[dl]^-{w_{L}'}\\
&1+\ltilde\bZ_l
}
\end{equation*}
where $\Gamma_{L',0} \hookrightarrow \Gamma_{L,0}$ is the natural homomorphism.''
Indeed, 
we set $L=K(\mu_{\ltilde})$ and $M_0=J^\Gamma$, 
then 
$L$ and $M_0$ satisfy the condition as above by Proposition \ref{1.10}.
For such a 
finite subextension $K_S/L/K$, 
$w_{L}'$ makes the following diagram
\begin{equation*}
\xymatrix{
\Gamma_{L(\mu_{\ltilde}),0}\ar@{^{(}-{>}}[rr] \ar[rd]_-{w_{L(\mu_{\ltilde})}'}\ar@{}[d]&&\Gamma_{L,0}\ar[dl]^-{w_{L}'}\\
&1+\ltilde\bZ_l
}
\end{equation*}
commute 
and by Proposition \ref{1.10}, $w_{L(\mu_{\ltilde})}'=w_{L(\mu_{\ltilde})}$.
Since $1+\ltilde\bZ_l$ is torsion free 
and 
the cokernel of $\Gamma_{L(\mu_{\ltilde}),0} \hookrightarrow \Gamma_{L,0}$ is finite, 
there exists a unique character of $\Gamma_{L,0}$ 
such that 
this diagram commutes, 
hence 
we obtain $w_{L}'=w_{L}$.
Similarly, $w_K$ can be characterized as the unique character of $\Gamma_{K,0}$  such that 
the following diagram commutes: 
\begin{equation*}
\xymatrix{
\Gamma_{L,0}\ar@{^{(}-{>}}[rr] \ar[rd]_-{w_L}\ar@{}[d]&&\Gamma_{K,0}\ar[dl]^-{}\\
&1+\ltilde\bZ_l.
}
\end{equation*}
\end{proof}

\begin{rem}\label{1.18}
In condition $(\star_l)$, 
we assume that 
for any $\bZ_l$-extension $K_{\infty}/K$, 
$K_{\infty}/K$ and $S$ satisfy condition $(\ast)$.
However, 
about condition $(\star_l)$ in Theorem \ref{1.17}, 
for 
non-cyclotomic 
$\bZ_l$-extensions $K_{\infty}/K$ 
which we can distinguish 
group-theoretically 
from $K_{\infty,0}/K$, 
it is not necessary to assume that 
$K_{\infty}/K$ and $S$ satisfy condition $(\ast)$.
Indeed, it is enough to modify 
the first sentence in the proof.

For example,
we say that $K_{\infty}/K$ satisfies condition $(\ddagger)$ 
if 
there exists a $\Lambda$-cofinite $\Lambda$-submodule $M_0 \subset X_{S}$ 
such that 
for any 
$(\gamma, \alpha)\in  (\Gamma \times (1+\ltilde\Bbb Z_l))^{\prim}$ 
and 
any $x\in M_0\setminus \{0\}$
with 
$\gamma-\alpha \in \Ann_{\Lambda}(x)$, 
we have 
$\gamma \in \Gamma^{\prim}$ 
and 
the index of $\langle \alpha \rangle$ in $1+\ltilde\Bbb Z_l$ is contant 
(i.e. does not depend on $(\gamma, \alpha)$ or $x$).
If $K_\infty = K_{\infty,0}$, 
set $M_0 \defeq J$, then 
$K_{\infty}/K$ satisfies condition $(\ddagger)$.
Indeed, 
as we have seen in the proof of Proposition \ref{1.9}, 
for any $(\gamma, \alpha)\in  (\Gamma \times (1+\ltilde\Bbb Z_l))^{\prim}$ 
and 
any $x\in J\setminus \{0\}$ 
with 
$\gamma-\alpha \in \Ann_{\Lambda}(x)$, 
there exist 
$\p \in (S\setminus \Sigma)$ and $\nu\in\Bbb Z_l$
such that 
$\gamma=\tilde\gamma_{\p}^{\nu}$ 
and 
$\alpha=w(\tilde\gamma_{\p})^{\nu}$.
Since 
$(\gamma, \alpha)\in  (\Gamma \times (1+\ltilde\Bbb Z_l))^{\prim}$, 
we have $\nu\in\Bbb Z_l^{\ast}$, 
hence 
$\gamma=\tilde\gamma_{\p}^{\nu}\in \Gamma^{\prim}$.
Further, 
the index of $\langle \alpha \rangle = \Image (w)$ in $1+\ltilde\Bbb Z_l$ is 
constant.
Note that 
we can 
distinguish group-theoretically whether or not 
the $\bZ_l$-extension $K_\infty/K$ 
corresponding to 
a $\bZ_l$-quotient of $G_{K,S}$ 
satisfies condition $(\ddagger)$.
We say that $S$ satisfies condition $(\star'_l)$
if for any $\bZ_l$-extension $K_{\infty}/K$ 
with $K_{\infty}L/L$ and $S(L)$ satisfying conditions $(\dagger')$ and $(\ddagger)$ for any finite subextension $K_S/L/K$, 
$K_{\infty}/K$ and $S$ satisfy condition $(\ast)$.
The assertion of 
Theorem \ref{1.17} is true even if we replace condition $(\star_l)$ with the weaker condition $(\star'_l)$.

Further, 
if we can distinguish group-theoretically 
the quotient $\Gamma_0$ of $G_{K,S}$ 
from 
the other 
$\bZ_l$-quotients, 
we can replace condition $(\star_l)$ in Theorem \ref{1.17} with ``$\# S= \infty$''.
\end{rem}

The following proposition gives a sufficient condition for condition $(\star_l)$.

\begin{prop}\label{1.20}
Assume that 
$\delta(S) \neq 0$ (cf. Notations).
Then any $\bZ_l$-extension $K_{\infty}/K$ 
and $S$ 
satisfy condition $(\ast)$.
In particular, $S$ satisfies condition $(\star_l)$.
\end{prop}
\begin{proof}
For 
$T\subset (S\setminus \Sigma)$ 
with
$\#((S\setminus \Sigma)\setminus T)<\infty$, 
$\delta_{\sup}(T)=\delta_{\sup}(S) > 0$.
Hence, 
by the Chebotarev density theorem for infinite extensions 
(\cite{Serre}, Chapter I, 2.2, COROLLARY 2), 
the closed subgroup of $G(K_{\infty}K_{\infty,0}/K)$ generated by the Frobenius elements at primes in $T$ 
has positive Haar measure in $G(K_{\infty}K_{\infty,0}/K)$, 
so that 
it is an open subgroup.
\end{proof}

\section{Local correspondence and recovering various local invariants}

In this section, we obtain 
the ``local correspondence'', 
and study its properties.


\begin{prop}\label{2.0}
Let $K$ be a number field, and $S$ a set of primes of $K$ 
with $P_{K,\infty} \subset S$.
Assume 
that $\# P_{S,f} \geq 2$.
Then, for $\overline{\p}, \overline{\q} \in S_f(K_S)$, $D_{\overline{\p}}(K_S/K) = D_{\overline{\q}}(K_S/K)$ if and only if $\overline{\p} = \overline{\q}$.
Let $l \in P_{S,f}$.
Then 
the set 
$\{ D_{\overline{\p}} \subset G_{K,S} \mid \overline{\p}\in S_f(K_S) \}$
is characterized group-theoretically from $G_{K,S}$ and the character $G_{K,S}\twoheadrightarrow \Gamma_0 \overset{w}\rightarrow 1+\ltilde\bZ_l$.
In particular, if $S$ satisfies condition $(\star_l)$, 
then 
$\{ D_{\overline{\p}} \subset G_{K,S} \mid \overline{\p}\in S_f(K_S) \}$
is characterized group-theoretically from $G_{K,S}$ (and $l$).
\end{prop}
\begin{proof}
The first assertion follows immediately from \cite{Ivanov}, Corollary 2.7(ii).
To prove the second 
assertion, 
we modify a little the reconstruction algorithm of (ii)$\rightsquigarrow$(i) in \cite{Ivanov}, Theorem 1.1\footnote{
In the proof of (ii)$\rightsquigarrow$(i) in \cite{Ivanov}, Theorem 1.1, 
where the finiteness of $S$ is assumed, 
we have to take an open subgroup $U_0$ of $G_{K,S}$ corresponding to a field which contains the $l$-th roots of unity and is totally imaginary, as in the last paragraph of \S 3 of \cite{Ivanov}.
If $S$ is finite, 
$\Ker(G_{K,S} \twoheadrightarrow G_{K,S}^{\ab}/(G_{K,S}^{\ab})^
{\phi(\ltilde)}
)$ satisfies the above condition, 
where $\phi$ is Euler's totient function.
On the other hand, if $S$ is not finite, 
the author does not know how to take such an open subgroup group-theoretically from $G_{K,S}$. 
In the rest of the proof, the finiteness of $S$ is not used.
} 
as follows.
Let $U_0$ be any open 
subgroup of $G_{K,S}$, 
corresponding to $K_{S}/L_0/K$.
Write $\chi_{L_0}^{(l),\naive} : U_0 \hookrightarrow G_{K,S}\twoheadrightarrow \Gamma_{K,0} \overset{w}\rightarrow 1+\ltilde\bZ_l\hookrightarrow{\bZ_l}^{\ast}$. Note that $\chi_{L_0}^{(l),\naive} = \chi_{L_0}^{(l)}$ if (and only if) $U_0 \subset G_{K(\mu_{\ltilde}),S(K(\mu_{\ltilde}))}$.
For each $U_0$, by 
using $\chi_{L_0}^{(l),\naive}$ 
instead of $\chi_{L_0}^{(l)}$ 
to run if possible 
the reconstruction algorithm as in the last paragraph of \S 3 of \cite{Ivanov} and \cite{Ivanov}, Remark 3.6, 
we obtain a set $\Dec_{U_0}^{\naive}$ consisting of subgroups of $G_{K,S}$ (if not possible, define $\Dec_{U_0}^{\naive}$ to be $\emptyset$ for convenience).
By exactly the same proof of (ii)$\rightsquigarrow$(i) in \cite{Ivanov}, Theorem 1.1, 
we have 
$\Dec_{U_0}^{\naive}=\{ D_{\overline{\p}} \subset G_{K,S} \mid \overline{\p}\in S_f(K_S) \}$
if $U_0 \subset G_{K(\mu_{\ltilde}),S(K(\mu_{\ltilde}))}$.
Thus, by taking $U_0$ such that 
$\Dec_{U_0}^{\naive}=\Dec_{U_0'}^{\naive}$ for any open subgroup $U_0'$ of $U_0$, 
we obtain $\Dec_{U_0}^{\naive} =\{ D_{\overline{\p}} \subset G_{K,S} \mid \overline{\p}\in S_f(K_S) \}$.

The last assertion follows from the second and Theorem \ref{1.17}.
\end{proof}

\begin{prop}(\cite{Chenevier-Clozel}, Remarque 5.3)\label{2.4}
Let $K$ be a totally real number field and $S$ a set of primes of $K$.
Assume that 
there exists a prime number $l$
with 
$P_{l} \cup P_\infty \subset S$.
Then 
the decomposition groups in $G_{K,S}$ 
at primes in $(S_f \setminus P_l)(K_S)$ are full.
\end{prop}

Note that 
with the notations 
in Proposition \ref{2.4}, 
if $\# P_{S,f} \geq 2$, 
then 
the decomposition groups in $G_{K,S}$ 
at primes in $S_f(K_S)$ are full.

For the following two lemmas, 
we use the theory of 
groups of $l$-decomposition type 
(cf. \cite{Ivanov}, \S 2).

\begin{lem}\label{2.1}
For $i=1,2$, let $p_i$ be a prime number, $\kappa_i$ a $p_i$-adic field and $\lambda_1/\kappa_1$ a Galois extension. 
Assume that 
there exists an 
isomorphism 
$\sigma :G(\lambda_1/\kappa_1) \isom G_{\kappa_2}$.
Then 
the residue characteristics and the residual degrees of $\kappa_1$ and $\kappa_2$ coincide, respectively, 
and $\sigma$ induces a bijection between the sets of Frobenius lifts.
Further, 
$[\kappa_1:\bQ_{p_1}]$ is 
greater than or equal to 
$[\kappa_2:\bQ_{p_2}]$.
\end{lem}
\begin{proof}
By local class field theory, 
for $i=1,2$, 
the residue characteristic $p_i$ of $\kappa_i$ 
can be characterized as the unique prime $l$ 
such that 
there exists a surjection $G_{\kappa_i} \twoheadrightarrow \bZ_l^2$.
Thus, 
we have $p_1=p_2$, 
and 
write $p$ for them.
We write $\phi$ for the composite of the canonical surjection 
$G_{\kappa_1} \twoheadrightarrow G(\lambda_1/\kappa_1)$ 
with $\sigma$.
For any prime $l$ different from $p$, 
by \cite{Ivanov}, 2.2, Local situation, 
any $l$-Sylow subgroup $G_{\kappa_1,l}$ of $G_{\kappa_1}$ 
is of $l$-decomposition type. 
$\phi(G_{\kappa_1,l})$ 
is an $l$-Sylow subgroup of $G_{\kappa_2}$, 
hence is of $l$-decomposition type.
Therefore, by \cite{Ivanov}, Lemma 2.2, 
$\phi|_{G_{\kappa_1,l}}$ is injective.
Thus, 
$\Ker(\phi)$ is a pro-$p$ group, 
hence, 
by \cite{NSW}, (7.5.7) Corollary, 
is contained in $V_{\kappa_1}$.
Therefore, again 
by \cite{NSW}, (7.5.7) Corollary, 
$\phi(V_{\kappa_1})$ is the maximal normal pro-$p$ subgroup of $G_{\kappa_2}$, 
and hence 
$\phi(V_{\kappa_1})=V_{\kappa_2}$.
Therefore, we have the following commutative diagram 
\begin{equation*}
\xymatrix{
G_{\kappa_1}\ar@{->>}[r] \ar@{->>}[rd] & G(\lambda_1/\kappa_1)\ar[r]^-{\simeq}_-{\sigma} \ar@{->>}[d] & G_{\kappa_2} \ar@{->>}[d] \\
& G_{\kappa_1}^{\tr} \ar[r]^-{\simeq}_-{} & G_{\kappa_2}^{\tr}.
}
\end{equation*}
For a 
$p$-adic field $\kappa$, 
by local class field theory, 
the order of the residue field of $\kappa$ 
is equal to 
$\# ((G_{\kappa}^{\tr,\ab})_{\tor}/(G_{\kappa}^{\tr,\ab}[p^{\infty}])) + 1$.
Thus, 
the orders of the residue fields of 
$\kappa_1$ and $\kappa_2$ coincide, 
so that 
their residual degrees also coincide.
Further, the subgroup $I_{\kappa}/V_{\kappa} \subset G_{\kappa}^{\tr}$ can be characterized group-theoretically from $G_{\kappa}^{\tr}$ by \cite{NSW}, (7.5.7) Corollary, 
and the sets of Frobenius lifts of $G_{\kappa}^{\tr}$ from $G_{\kappa}^{\tr}$ and the order of the residue field of $\kappa$ 
by \cite{NSW}, (12.1.8) Lemma.
Therefore, $\sigma$ induces  a bijection between the sets of Frobenius lifts.
Finally, 
by local class field theory, 
$G_{\kappa}^{\ab,(p),/\tor}$ is a free $\bZ_p$-module of rank $[\kappa:\bQ_p] + 1$.
The last assertion follows from this.
\end{proof}

About Frobenius elements, 
the following lemma is also useful.

\begin{lem}\label{2.1.5}
Let $p$, $l$ be different prime numbers, $\kappa$ a $p$-adic field and $\lambda/\kappa$ a Galois extension. 
Then 
$\kappa(\mu_l)^{(l)} \subset \lambda$ 
if and only if 
there exists an open subgroup $U$ of $G(\lambda/\kappa)$ 
such that $U^{(l)}$ is of $l$-decomposition type.
Assume that 
these conditions hold.
Then the surjection 
$G(\lambda/\kappa) \twoheadrightarrow (G(\lambda/\kappa)/I(\lambda/\kappa))^{(l)} ( \simeq \bZ_l )$
is characterized group-theoretically from $G(\lambda/\kappa)$.
Further, 
the data of 
the Frobenius element in $(G(\lambda/\kappa)/I(\lambda/\kappa))^{(l)}$
is equivalent to the data of 
the residue characteristic 
and the residual degree of $\kappa$.
\end{lem}
\begin{proof}
If $\kappa(\mu_l)^{(l)} \subset \lambda$, 
then, 
by \cite{NSW}, (7.5.9) Proposition, 
$G(\lambda/\kappa(\mu_l))^{(l)} = G(\kappa(\mu_l)^{(l)}/\kappa(\mu_l))$ 
is of $l$-decomposition type.
If $\kappa(\mu_l)^{(l)} \not\subset \lambda$, 
then, 
for any finite subextension $\kappa'$ of $\lambda/\kappa$ 
with $\mu_l \subset \kappa'$, 
the canonical surjection 
$G_{\kappa'}^{(l)} \twoheadrightarrow G(\lambda/\kappa')^{(l)}$ 
is not an isomorphism.
Hence, by \cite{NSW}, (7.5.9) Proposition and \cite{Ivanov}, Lemma 2.2, 
$G(\lambda/\kappa')^{(l)}$ is not of $l$-decomposition type.
For a finite subextension $\kappa'$ of $\lambda/\kappa$ with $\mu_l \not\subset \kappa'$, by \cite{NSW}, (7.5.9) Proposition, 
$G(\lambda/\kappa')^{(l)}$ is not of $l$-decomposition type.

Assume that $\kappa(\mu_l)^{(l)} \subset \lambda$.
By \cite{NSW}, (7.5.9) Proposition, 
the surjection 
$G(\lambda/\kappa) \twoheadrightarrow (G(\lambda/\kappa)/I(\lambda/\kappa))^{(l)}
$
is the unique $\bZ_l$-quotient of $G(\lambda/\kappa)$, 
and 
the open subgroup $G(\lambda/\kappa(\mu_l))$ of $G(\lambda/\kappa)$ is the maximal open subgroup of $G(\lambda/\kappa)$ 
whose maximal pro-$l$ quotient is 
of $l$-decomposition type.
Now, we set $n \defeq (G(\lambda/\kappa):G(\lambda/\kappa(\mu_l)))$, 
$G \defeq G(\lambda/\kappa(\mu_l))^{(l)} = G(\kappa(\mu_l)^{(l)}/\kappa(\mu_l))$ 
and 
$I \defeq I(\kappa(\mu_l)^{(l)}/\kappa(\mu_l))$, 
and write $f$ for the residual degree of $\kappa$.
By \cite{NSW}, (7.5.9) Proposition, 
$G$ 
is of $l$-decomposition type.
Hence, by \cite{Ivanov}, Lemma 2.2, 
the subgroup 
$I$ 
of 
$G$ 
is characterized group-theoretically from 
$G$, 
and 
$G$ 
is a semi-direct product of 
$G/I$
by 
$I$.
Further, by \cite{NSW}, (7.5.2) Proposition, 
the character defining this semi-direct product 
is an injection 
$G/I \to \Aut(I) = {\bZ_l}^{\ast}$
under which the Frobenius element 
$\Frob_{\kappa(\mu_l)}$
maps to $p^{nf}$.
Finally, under the canonical homomorphism 
$G/I \to (G(\lambda/\kappa)/I(\lambda/\kappa))^{(l)}$
induced by the inclusion $G(\lambda/\kappa(\mu_l)) \hookrightarrow G(\lambda/\kappa)$, 
$\Frob_{\kappa(\mu_l)}$ maps to $\Frob_{\kappa}^n$.
Since 
$(G(\lambda/\kappa)/I(\lambda/\kappa))^{(l)} \simeq \bZ_l$ 
is torsion free, 
we obtain the last assertion.
\end{proof}

Note that, 
by Lemma \ref{2.1.5}, 
it is possible to distinguish whether or not $\kappa(\mu_l)^{(l)} \subset \lambda$ group-theoretically from $G(\lambda/\kappa)$.

\begin{definition}
For $i=1,2$, 
let $K_i$ be a number field, $S_i$ a set of primes of $K_i$, 
$T_i \subset S_{i,f}$, 
and $\sigma :G_{K_1,S_1}\isom G_{K_2,S_2}$ 
an 
isomorphism.
We say that 
the local correspondence 
between $T_1$ and $T_2$ 
holds for $\sigma$, 
if the following conditions are satisfied: 
\begin{itemize}

\item[$\bullet$]
For any $\overline{\p}_1 \in T_1(K_{1,S_1})$, 
there is a unique prime $\sigma_{\ast}(\overline{\p}_1) \in T_2(K_{2,S_2})$ 
with $\sigma(D_{\overline{\p}_1}) = D_{\sigma_{\ast}(\overline{\p}_1)}$, such that 
$\sigma_{\ast} \colon T_1(K_{1,S_1})\to T_2(K_{2,S_2})$, 
$\overline{\p}_1 \mapsto \sigma_{\ast}(\overline{\p}_1)$ 
is a bijection. 

\end{itemize}

Then 
$\sigma_{\ast}$ is Galois-equivariant, 
i.e., 
for each $g \in G_{K_1,S_1}$ 
and 
$\bap_1 \in T_1(K_{1,S_1})$, 
$\sigma_{\ast}(g \bap_1) = \sigma(g)\sigma_{\ast}(\bap_1)$.
Further, 
for any open subgroup $U_1$ of $G_{K_1,S_1}$, we set $U_2 \defeq \sigma(U_1)$ and for $i=1,2$, write $K_{i,{S_i}}/L_i/K_i$ for the subextension corresponding to $U_i$.
Two primes $\bap_1$, $\baq_1 \in T_{1}(K_{1,S_1})$ restrict to the same prime of $L_1$ (this condition is equivalent to the condition that $D_{\bap_1}$ and $D_{\baq_1}$ are conjugate by an element in $U_1$), 
if and only if 
$\sigma_{\ast}(\bap_1), \sigma_{\ast}(\baq_1) \in T_2(K_{2,S_2})$ 
restrict to the same prime of $L_2$, 
and hence $\sigma_{\ast}$
induces a bijection
$\sigma_{\ast,U_1}=\sigma_{\ast,L_1} \colon T_1(L_1) \isom T_2(L_2)$.

Moreover, 
we say that 
the good local correspondence 
between $T_1$ and $T_2$ 
holds for $\sigma$, 
if the following conditions are satisfied: 
\begin{itemize}
\item[$\bullet$]
The local correspondence between $T_1$ and $T_2$ 
holds for $\sigma$.

\item[$\bullet$]
$\sigma_{\ast,K_1}$ preserves the residue characteristics and
the residual degrees 
of all primes in $T_1$.
\end{itemize}
\end{definition}

\begin{theorem}\label{2.5}
For $i=1,2$, 
let $K_i$ be a number field, $S_i$ a set of primes of $K_i$ 
with $P_{K_i,\infty} \subset S_i$ 
and 
$\sigma :G_{K_1,S_1}\isom G_{K_2,S_2}$ an isomorphism.
Assume that 
$\# P_{S_i,f} \geq 2$ for $i=1,2$ 
and 
that there exists 
a prime number $l \in P_{S_1,f} \cap P_{S_2,f}$ 
such that 
$S_i$ satisfies condition $(\star_l)$ for $i=1,2$.
Then 
the local correspondence between $S_{1,f}$ and $S_{2,f}$
holds for $\sigma$.
Further, 
let $T_1 \subset S_{1,f}$ and $T_2 \subset S_{2,f}$ be subsets 
between which the local correspondence holds for $\sigma$
and 
assume 
that 
for 
one $i$,
there exist a totally real subfield $K_{i,0} \subset K_i$ and a set of primes $T_{i,0}$ of $K_{i,0}$ 
such that $T_{i,0}(K_i)=T_i$.
Then 
the good local correspondence between $T_1$ and $T_2$
holds for $\sigma$.
\end{theorem}
\begin{proof}
By Proposition \ref{2.0}, 
the local correspondence between $S_{1,f}$ and $S_{2,f}$
holds for $\sigma$.
Let $i$ be the $i$ which appears in the last assumption.
Take two distinct primes $p, p' \in P_{S_i,f}$.
Then, by Proposition \ref{2.4}, the decomposition groups in 
$G_{K_{i,0},T_{i,0} \cup P_{p} \cup P_{p'} \cup P_{\infty}}$ 
at primes in $T_{i,0}(K_{i,0,T_{i,0} \cup P_{p} \cup P_{p'} \cup P_{\infty}})$ 
are full.
Since 
$K_{i,T_i \cup P_{p} \cup P_{p'} \cup P_{\infty}}=K_{i,0,T_{i,0} \cup P_{p} \cup P_{p'} \cup P_{\infty}}$, 
the decomposition groups in $G_{K_i,T_i \cup P_{p} \cup P_{p'} \cup P_{\infty}}$ 
at primes in $T_i(K_{i,T_i \cup P_{p} \cup P_{p'} \cup P_{\infty}})$ are full, 
so that 
the decomposition groups in $G_{K_i,S_i}$ 
at primes in $T_i(K_{i,S_i})$ are full.
Thus,
by Lemma \ref{2.1}, 
the good local correspondence between $T_1$ and $T_2$
holds for $\sigma$.
\end{proof}

\begin{rem}\label{2.6}
With the notations as in Theorem \ref{2.5}, for $i=1,2$, let $U_i$ be an open subgroup of $G_{K_i,S_i}$ with $\sigma(U_1) = U_2$, and $L_i$ the subextension of $K_{i,{S_i}}/K_i$ corresponding to $U_i$.
Then 
the local correspondence (resp. the good local correspondence) between $S_{1,f}(L_1)$ (resp. $T_1(L_1)$) and $S_{2,f}(L_2)$ (resp. $T_2(L_2)$) 
holds for $\sigma|_{U_1} : U_1\isom U_2$.
Indeed, 
the assertion follows from Theorem \ref{2.5}, \cite{Ivanov}, Proposition 2.4
and the fact that 
for $i=1,2$ and for $\bap_i \in S_{i,f}(K_{i,S_i})$, $D_{\bap_i, K_{i,{S_i}}/L_i} = D_{\bap_i, K_{i,{S_i}}/K_i} \cap U_i$.
\end{rem}

\begin{prop}\label{2.7}
For $i=1,2$, 
let $K_i$ be a number field, $S_i$ a set of primes of $K_i$ 
with $P_{K_i,\infty} \subset S_i$, 
$T_i \subset S_{i,f}$ 
and $\sigma :G_{K_1,S_1}\isom G_{K_2,S_2}$ 
an isomorphism.
Assume that 
$\# P_{T_i} \geq 2$ for $i=1,2$ 
and 
that 
the 
local correspondence between $T_1$ and $T_2$
holds for $\sigma$.
Then 
$[K_1 : \bQ]=[K_2:\bQ]$, 
$P_{T_1}=P_{T_2}$, 
$P_{T_1} \cap \cs(K_1/\bQ) = P_{T_2} \cap \cs(K_2/\bQ)$, 
and the good local correspondence between $P_{T_1}(K_1)$ 
and $P_{T_2}(K_2)$ 
holds for $\sigma$.
\end{prop}
\begin{proof}
By Proposition \ref{2.4}, 
for $i=1,2$, 
the decomposition group in $G_{K_i,S_i}$ 
at any prime in $P_{T_i}(K_{i,S_i})$ is full. 
Take $l_1 \in P_{T_1}$ and $l_2 \in P_{T_2}$.
By
Lemma \ref{2.1}, 
$\sigma_{\ast,K_1}$ induces a bijection between 
$P_{K_1,l_1}$ and $T_2 \cap P_{K_2,l_1}$, 
and we have 
$[K_{1,\fp_1}:\bQ_{l_1}] \leq [K_{2,\sigma_{\ast,K_1}(\fp_1)}:\bQ_{l_1}]$
for $\fp_1 \in P_{K_1,l_1}$.
Therefore, 
\begin{equation*}
\begin{split} 
[K_1:\bQ] &= \sum_{\fp_1 \in P_{K_1,l_1}} [K_{1,\fp_1}:\bQ_{l_1}] \leq \sum_{\fp_1 \in P_{K_1,l_1}} [K_{2,\sigma_{\ast,K_1}(\fp_1)}:\bQ_{l_1}] = \sum_{\fp_2 \in T_2 \cap P_{K_2,l_1}} [K_{2,\fp_2}:\bQ_{l_1}] \\ &\leq \sum_{\fp_2 \in P_{K_2,l_1}} [K_{2,\fp_2}:\bQ_{l_1}] = [K_2:\bQ].
\end{split}
\end{equation*}
By the same argument 
obtained by exchanging the roles of $1$ and $2$, 
we have 
$[K_2:\bQ] \leq [K_1:\bQ]$.
Thus, 
we obtain 
$[K_1 : \bQ]=[K_2:\bQ]$, 
$P_{K_1,l_2} \subset T_1$ and $P_{K_2,l_1} \subset T_2$, 
so that 
$P_{T_1}=P_{T_2}$ 
and 
$\sigma_{\ast,K_1}$ induces a bijection between 
$P_{T_1}(K_1)$ and $P_{T_2}(K_2)$.
Hence by Lemma \ref{2.1}, 
the good local correspondence between $P_{T_1}(K_1)$ 
and $P_{T_2}(K_2)$ 
holds for $\sigma$.
Again by Lemma \ref{2.1}, 
$\sigma_{\ast,K_1}|_{P_{T_1}(K_1)} : P_{T_1}(K_1)\isom P_{T_2}(K_2)$
also preserves 
the local degrees 
of all primes in $P_{T_1}(K_1)$.
Therefore, we have $P_{T_1} \cap \cs(K_1/\bQ) = P_{T_2} \cap \cs(K_2/\bQ)$.
\end{proof}

\section{The existence of an isomorphism of fields}

In this section, we 
develop a way to
show the existence of an isomorphism between 
the number fields in question 
assuming that the good local correspondence holds.

Let $K$ be a number field, and
$l$ a prime number.
We set 
$
\Gamma_K = \Gamma_{K, l} \defeq G_{K}^{\ab,(l),/\tor}$, 
$r_l(K) \defeq \rank_{\bZ_l}\Gamma_K$, 
and write $K^{(\infty)}=K^{(\infty,l)}$ for the extension of $K$ corresponding to $\Gamma_K$.
We write $\fd_l(K)(\geq 0)$ for the Leopoldt defect.
Then, by \cite{NSW}, (10.3.20) Proposition, 
we have $r_l(K) = r_\bC(K)+1+\fd_l(K) \leq [K:\bQ]$, where $r_\bC(K)$ is the number of complex primes of $K$.
For $\p \in P_{K,f} \setminus P_{K,l}$, $\p$ is unramified in $K^{(\infty)}/K$.
Then 
we write $\Frob_{\p}$ for the Frobenius element in $\Gamma_K$ at $\p$.

We write 
$\Hom_{cts}(\Gamma_{K}, \bZ_l) (\simeq \Hom_{cts}(G_{K}, \bZ_l))$ for the set of continuous homomorphisms from $\Gamma_{K}$ to $\bZ_l$, 
where $\bZ_l$ is equipped with the profinite topology.
Then \begin{equation*}\Hom_{cts}(\Gamma_{K}, \bZ_l) = \Hom_{\bZ_l}(\Gamma_{K}, \bZ_l) \simeq \bZ_l^{r_l(K)}.\end{equation*}
Note that $\Hom_{cts}(\Gamma_{K}, \bZ_l)$ coincides with the continuous cochain cohomology of degree $1$ (cf. \cite{NSW}, Chapter \hspace{-.1em}I\hspace{-.1em}I, \S 6).

\begin{lem}\label{4f.1}
Let $K$ be a number field, $L$ a finite Galois extension of $K$, and $l$ a prime number.
Then the 
homomorphism: 
$\Hom_{cts}(\Gamma_{K}, \bZ_l) \to \Hom_{cts}(\Gamma_{L}, \bZ_l)^{G(L/K)}$ 
induced from the restriction map 
is injective, and 
the cokernel of this homomorphism
is finite.
In particular, $\rank_{\bZ_l} \Hom_{cts}(\Gamma_{L}, \bZ_l)^{G(L/K)} = r_l(K)$.
\end{lem}
\begin{proof}
For $n \in \bZ_{>0}$, by \cite{NSW}, (1.6.7) Proposition, 
there exists an exact sequence
\begin{equation*}0\to H^1(G(L/K),\bZ/l^n\bZ) 
\to H^1(G_{K},\bZ/l^n\bZ) 
\to H^1(G_{L},\bZ/l^n\bZ)^{G(L/K)} 
\to H^2(G(L/K),\bZ/l^n\bZ) .\end{equation*}
Since $H^1(G(L/K),\bZ/l^n\bZ)$ is finite, by passing to the projective limit, 
we have an exact sequence
\begin{equation*}0\to \Hom_{cts}(G(L/K), \bZ_l) 
\to \Hom_{cts}(\Gamma_{K}, \bZ_l) 
\to \Hom_{cts}(\Gamma_{L}, \bZ_l)^{G(L/K)} 
\to 
H^2(G(L/K),\bZ_l).\end{equation*}
Since 
$\Hom_{cts}(G(L/K), \bZ_l)$ is trivial and $H^2(G(L/K),\bZ_l) \simeq H^1(G(L/K),\bQ_l/\bZ_l)$ is finite, we obtain the assertion.
\end{proof}

\begin{lem}\label{4f.2}
Let $K$ be a number field, $L$ a finite extension of $K$, and $l$ a prime number.
Assume that $L \neq K$ and $L$ has a complex prime. 
Then we have $r_l(K) < r_l(L)$.
\end{lem}
\begin{proof}
By the assumption, we have 
$r_\bC(K) < r_\bC(L)$.
Further, by \cite{NSW}, (10.3.7) Corollary and the proof of \cite{NSW}, (10.3.11) Corollary, we have 
$\fd_l(K) \leq \fd_l(L)$.
\end{proof}

In the rest of this 
paper, 
fix an algebraic closure $\overline{\bQ}$ of $\bQ$, 
and 
suppose that 
all number fields and all algebraic extensions of them 
are subfields of $\overline{\bQ}$.

Let $L$ be a finite extension of $K$.
We have the canonical homomorphisms $\pi_{L/K} = \pi_{L/K, l}:\Gamma_{L} \to \Gamma_{K}$ and $res_L^{K} = res_{L,l}^{K} : \Hom_{cts}(\Gamma_{K}, \bZ_l) \to \Hom_{cts}(\Gamma_{L}, \bZ_l)$.

\begin{prop}\label{4f.3}
For $i=1,2$, 
let $K_i$ be a number field, $S_i$ a set of primes of $K_i$ with $P_{K_i,\infty} \subset S_i$, 
$T_i \subset S_{i,f}$ and $\sigma :G_{K_1,S_1}\isom G_{K_2,S_2}$ an isomorphism.
Assume that the following conditions hold:
\begin{itemize}

\item[(a)]
The good local correspondence between $T_1$ and $T_2$
holds for $\sigma$.

\item[(b)]
There exists a finite extension $L/K_1K_2$ such that 
$L/\bQ$ is Galois and 
$\delta(T_i(L)) \neq 0$ for one $i$.

\end{itemize}
Then, 
for any prime number $l \in P_{S_1,f} \cap P_{S_2,f}$, 
there exists $\tau \in G(L/\bQ)$ such that 
\begin{equation*}\overline{\sigma} \circ \pi_{L/K_1, l} = (\tau|_{K_2})^\ast \circ \pi_{L/\tau(K_2), l},\end{equation*}
where $\overline{\sigma}:\Gamma_{K_1, l} \isom \Gamma_{K_2, l}$ is the isomorphism induced by $\sigma$ 
and 
$(\tau|_{K_2})^\ast:\Gamma_{\tau(K_2), l} \isom \Gamma_{K_2, l}$ is the isomorphism induced by $\tau|_{K_2}$.
\end{prop}
\begin{proof}
Take $l \in P_{S_1,f} \cap P_{S_2,f}$.
We have the following diagram:
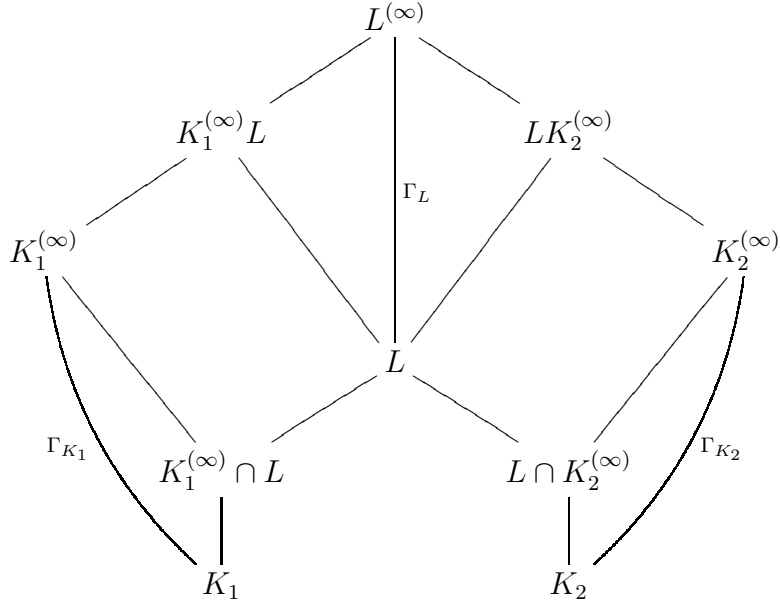
\begin{equation*}
\xymatrix{
& & L^{(\infty)} \ar@{-}[ld] \ar@{-}[ddd]^{\Gamma_{L}} \ar@{-}[rd] & & \\
& K_1^{(\infty)}L \ar@{-}[ld] \ar@{-}[rdd] & & LK_2^{(\infty)} \ar@{-}[ldd] \ar@{-}[rd] & \\
K_1^{(\infty)} \ar@{-}[rdd]^{} \ar@/_15pt/@{-}[rddd]_{\Gamma_{K_1}} & & & & K_2^{(\infty)} \ar@{-}[ldd]_{} \ar@/^15pt/@{-}[lddd]^{\Gamma_{K_2}} \\
& & L \ar@{-}[ld] \ar@{-}[rd] & & \\
& K_1^{(\infty)} \cap L \ar@{-}[d] & & L \cap K_2^{(\infty)} \ar@{-}[d] & \\
& K_1 & & K_2 & \\
}
\end{equation*}
For each $\tau \in G(L/\bQ)$, 
let $\tau^\ast \in \Aut(\Gamma_{L})$ be 
the automorphism of $\Gamma_{L}$ 
defined by the 
outer 
action of $\tau$.
We set 
$\phi_\tau 
:\Gamma_{L} \to \Gamma_{K_2}$, 
$\gamma \mapsto (\overline{\sigma} \circ \pi_{L/K_1})(\gamma) \cdot (\pi_{L/K_2} \circ \tau^\ast)(\gamma)^{-1}$, 
then this is a homomorphism of free $\bZ_l$-modules.

Assume that for any $\tau \in G(L/\bQ)$, 
$\overline{\sigma} \circ \pi_{L/K_1} \neq \pi_{L/K_2} \circ \tau^\ast$.
Then 
$\rank_{\bZ_l}\Ker(\phi_\tau) < r_l(L)$, so that 
$\Ker(\phi_\tau)$ has Haar measure $0$ in $\Gamma_{L}$.
Hence 
$\bigcup_{\tau \in G(L/\bQ)} \Ker(\phi_\tau)$ also has Haar measure $0$.
Now, by symmetry, we may assume that 
$\delta_{\sup}(T_1(L) \cap \cs(L/\bQ)(L)) = \delta_{\sup}(T_1(L)) > 0$.
Then, 
by the Chebotarev density theorem for infinite extensions 
(\cite{Serre}, Chapter I, 2.2, COROLLARY 2), 
there exists 
$\p \in (T_1 \setminus P_{K_1,l})(L) \cap \cs(L/\bQ)(L)$ 
such that 
$\Frob_\p$ in $\Gamma_{L}$ 
satisfies 
$\Frob_\p \notin \bigcup_{\tau \in G(L/\bQ)} \Ker(\phi_\tau)$.
For $i=1,2$, 
we set 
$\p_i \defeq \p|_{K_i}$, 
then 
$\pi_{L/K_i}(\Frob_{\p})=\Frob_{\p_i}$.
When 
$\p$ is above a prime number $p$, 
$\p_1$ and $\p_2$ are also above $p$.
Since the good local correspondence between $T_1$ and $T_2$
holds for $\sigma$, 
$\sigma_{\ast,K_1}(\p_1)$ is also above $p$.
Therefore, there exists $\tau' \in G(L/\bQ)$ such that 
$\sigma_{\ast,K_1}(\p_1)=(\tau' \cdot \p)|_{K_2}$.
Thus, we have $\pi_{L/K_2} \circ \tau'^\ast (\Frob_{\p}) = \Frob_{\sigma_{\ast,K_1}(\p_1)}$.
By the proof of \cite{Ivanov}, Lemma 2.3, 
$D_{\p_1} \subset G_{K_1,S_1}$ 
and $D_{\sigma_{\ast,K_1}(\p_1)} \subset G_{K_2,S_2}$ 
satisfy the conditions in Lemma \ref{2.1.5} for $p$ and $l$ above.
Hence, by Lemma \ref{2.1.5} and the good local correspondence between $T_1$ and $T_2$ 
for $\sigma$, 
we obtain $\overline{\sigma}(\Frob_{\p_1})=\Frob_{\sigma_{\ast,K_1}(\p_1)}$. 
Thus, we obtain 
\begin{equation*}
\begin{split}
\overline{\sigma} \circ \pi_{L/K_1}(\Frob_\p)
=\Frob_{\sigma_{\ast,K_1}(\p_1)}
=\pi_{L/K_2} \circ \tau'^\ast (\Frob_{\p}).
\end{split}
\end{equation*}
Namely, 
$\Frob_\p \in \Ker(\phi_{\tau'})$. 
However, this contradicts the fact that 
$\Frob_\p \notin \bigcup_{\tau \in G(L/\bQ)} \Ker(\phi_\tau)$.

Therefore, 
there exists $\tau \in G(L/\bQ)$ such that 
$\overline{\sigma} \circ \pi_{L/K_1} = \pi_{L/K_2} \circ \tau^\ast = {(\tau|_{K_2})^\ast} \circ \pi_{L/\tau(K_2)}$.
\end{proof}

\begin{lem}\label{4f.4}
For $i=1,2$, let $K_i$ be a number field, and $l$ a prime number.
Assume that 
there exist a finite extension $L/K_1K_2$ 
and an isomorphism $\overline{\sigma}: \Gamma_{K_1, l} \isom \Gamma_{K_2, l}$
such that 
$\overline{\sigma} \circ \pi_{L/K_1, l} = \pi_{L/K_2, l}$.
Then 
$K_1^{(\infty,l)}K_2 = K_1K_2^{(\infty,l)}$.
\end{lem}
\begin{proof}
We set 
$\Gamma \defeq G(K_1^{(\infty,l)}K_2^{(\infty,l)}/K_1K_2)$. 
Then we have the following diagram:
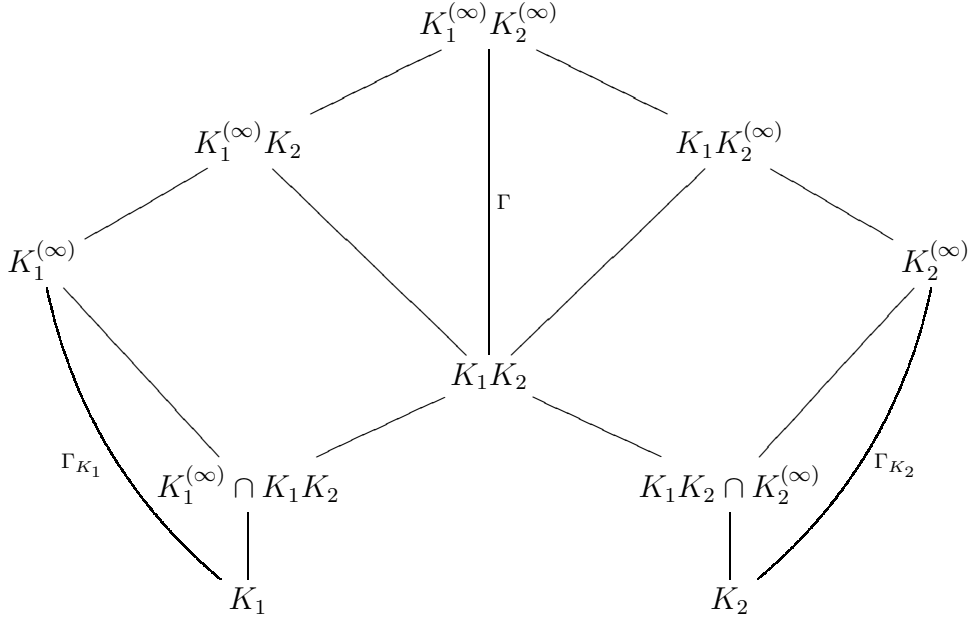
\begin{equation*}
\xymatrix{
& & K_1^{(\infty)}K_2^{(\infty)} \ar@{-}[ld] \ar@{-}[ddd]^{\Gamma} \ar@{-}[rd] & & \\
& K_1^{(\infty)}K_2 \ar@{-}[ld] \ar@{-}[rdd] & & K_1K_2^{(\infty)} \ar@{-}[ldd] \ar@{-}[rd] & \\
K_1^{(\infty)} \ar@{-}[rdd]^{} \ar@/_15pt/@{-}[rddd]_{\Gamma_{K_1}} & & & & K_2^{(\infty)} \ar@{-}[ldd]_{} \ar@/^15pt/@{-}[lddd]^{\Gamma_{K_2}} \\
& & K_1K_2 \ar@{-}[ld] \ar@{-}[rd] & & \\
& K_1^{(\infty)} \cap K_1K_2 \ar@{-}[d] & & K_1K_2 \cap K_2^{(\infty)} \ar@{-}[d] & \\
& K_1 & & K_2 & \\
}
\end{equation*}
We write $\pi_1$ for 
the composite of $\Gamma \twoheadrightarrow G(K_1^{(\infty)}K_2/K_1K_2) \overset{\text{restriction}}\isom G(K_1^{(\infty)}/K_1^{(\infty)} \cap K_1K_2) \hookrightarrow \Gamma_{K_1}$,
and define 
$\pi_2 : \Gamma \to \Gamma_{K_2}$ similarly.
Then 
$(\pi_1 , \pi_2) : \Gamma \to \Gamma_{K_1} \times \Gamma_{K_2}$ is injective.
We write $\pi$ for 
the canonical homomorphism: $\Gamma_{L} \to \Gamma$.
Then 
$\pi_{L/K_i} = \pi_i \circ \pi$ for $i=1,2$.
Since $\Image(\pi)$ is open in $\Gamma$, 
we obtain 
$\overline{\sigma} \circ \pi_1 = \pi_2$, 
so that 
$\Ker(\pi_2)=\Ker(\overline{\sigma} \circ \pi_1)=\Ker(\pi_1)$.
As 
$(\pi_1 , \pi_2)$ is injective, 
we have $\Ker(\pi_2)=\Ker(\pi_1)=1$.
Thus, we obtain $K_1^{(\infty)}K_2 = K_1^{(\infty)}K_2^{(\infty)} = K_1K_2^{(\infty)}$.
\end{proof}

\begin{prop}\label{4f.5}
For $i=1,2$, let $K_i$ be a number field.
Assume that the following conditions hold:
\begin{itemize}
\item[(a)]
$K_i$ has a complex prime 
for one $i$. 

\item[(b)]
There exists a finite extension $L/K_1K_2$ such that 
$L/\bQ$ is Galois and 
$K_1^{(\infty,l)}L = LK_2^{(\infty,l)}$ for a prime number $l$.

\end{itemize}
Then 
$K_1 = K_2$.
\end{prop}
\begin{proof}
Take a prime number $l$ for which $K_1^{(\infty,l)}L = LK_2^{(\infty,l)}$.
Then 
$\Ker(\pi_{L/K_1, l}) = \Ker(\pi_{L/K_2, l})$, 
for which we write $N$.
For $i=1,2$, 
$res_L^{K_i}$ factors through $\Hom_{cts}(\Gamma_{L}/N, \bZ_l) (\subset \Hom_{cts}(\Gamma_{L}, \bZ_l))$
and $\rank_{\bZ_l}\Hom_{cts}(\Gamma_{L}/N, \bZ_l) = r_l(K_i)$.
Hence 
$\Image (res_L^{K_1}) \cap \Image (res_L^{K_2}) \subset \Hom_{cts}(\Gamma_{L}/N, \bZ_l)$ 
and 
$\rank_{\bZ_l}(\Image (res_L^{K_1}) \cap \Image (res_L^{K_2})) = r_l(K_1) = r_l(K_2)$.
Therefore, by Lemma \ref{4f.1}, we have 
\begin{equation*}\Image (res_L^{K_1}) \cap \Image (res_L^{K_2}) 
\subset \Hom_{cts}(\Gamma_{L}, \bZ_l)^{G(L/K_1)} \cap \Hom_{cts}(\Gamma_{L}, \bZ_l)^{G(L/K_2)} = \Hom_{cts}(\Gamma_{L}, \bZ_l)^{G(L/K_1 \cap K_2)}\end{equation*}
and 
$r_l(K_1) 
= r_l(K_2) 
= r_l(K_1 \cap K_2)$.
Now, by symmetry, we may assume that $K_1$ has a complex prime.
Then, by Lemma \ref{4f.2}, we have $K_1 = K_1 \cap K_2$, so that $K_1 \subset K_2$.
Hence $K_2$ also has a complex prime.
Again by Lemma \ref{4f.2}, we have $K_2 = K_1$.
\end{proof}

\begin{prop}\label{4f.6}
Let $S$ and $T$ be subsets of $P_{\bQ,f}$ such that $\delta(S) \neq 0$ 
and $T$ is infinite.
Then there exists a finite subextension $L$ of $\bQ(\mu_{\ltilde} ; l \in T)/\bQ$ such that 
$L$ is totally imaginary, 
$[L:\bQ]$ is a power of $2$, and 
$\delta(S \cap \cs(L/\bQ)) \neq 0$.
\end{prop}
\begin{proof}
For $l \in T$, we write $\bQ(\alpha_l)$ for the maximal subextension of $\bQ(\mu_{\ltilde})/\bQ$ of $2$-power degree. 
We set $2^{n_l} \defeq [ \bQ(\alpha_l):\bQ ]$, $T = \{ l_1, l_2, l_3, \ldots \}$ with $n_{l_t} \leq n_{l_{t+1}}$ for $t \in \bZ_{>0}$, 
and 
$G_m \defeq G(\bQ(\alpha_{l_1}, \ldots ,\alpha_{l_m})/\bQ) 
= \prod_{1 \leq t \leq m} G(\bQ(\alpha_{l_t})/\bQ) \simeq \prod_{1 \leq t \leq m}\bZ/2^{n_{l_t}}\bZ$ (and by this isomorphism, identify $G_m$ with $\prod_{1 \leq t \leq m}\bZ/2^{n_{l_t}}\bZ$).
Note that the decomposition group $D_\infty$ at the real prime in $G_m$ is generated by the element $(2^{n_{l_1}-1}, \ldots ,2^{n_{l_m}-1})$.
We set 
\begin{equation*}X_m \defeq (2^{n_{l_t}-n_{l_1}})_{1 \leq t \leq m} + \prod_{1 \leq t \leq m}2^{n_{l_t}-n_{l_1}+1}\bZ/2^{n_{l_t}}\bZ.\end{equation*}
Then $\left\{ \langle x \rangle \subset G_m \mid 
x \in X_m \right\}$
coincides with 
the set of maximal elements of
\begin{equation*}\left\{ H \subset G_m \mid 
\text{$H$ is a cyclic subgroup of $G_m$ such that }
D_\infty \subset H \right\}.\end{equation*}
We set 
$Y_m \defeq \cup_{x \in X_m}\langle x \rangle$.
For $x \in X_m$, we have $\#\langle x \rangle = 2^{n_{l_1}}$.
Since 
$\# Y_m \leq \#X_m \cdot 2^{n_{l_1}} = (2^{n_{l_1}-1})^m \cdot 2^{n_{l_1}}$
and $\# G_m \geq 2^{n_{l_1}m}$, 
$\lim_{m \to \infty} \# Y_m / \# G_m = 0$.
Hence there exists $m' \in \bZ_{>0}$ such that $\delta_{\sup}(S) > \# Y_{m'} / \# G_{m'}$. 
Therefore, by the Chebotarev density theorem, 
there exists $x' \in G_{m'} \setminus Y_{m'}$ such that 
$\delta_{\sup}(S') > 0$
where \begin{equation*}S' \defeq 
\left\{ p \in S \setminus \Ram(\bQ(\alpha_{l_1}, \ldots ,\alpha_{l_{m'}})/\bQ) \mid 
\text{the Frobenius element at $p$ in $G_{m'}$ coincides with $x'$}
\right\}.\end{equation*}
Let $L$ be a finite subextension of $\bQ(\alpha_{l_1}, \ldots ,\alpha_{l_{m'}})/\bQ$ corresponding to $G_{m'} / \langle x' \rangle$.
Then, since $S' \subset \cs(L/\bQ)$, we have 
$\delta_{\sup}(S \cap \cs(L/\bQ)) \geq \delta_{\sup}(S') > 0$.
By construction, $L$ is totally imaginary, and 
$[L:\bQ]$ is a power of $2$.
\end{proof}

\section{Some properties of the Dirichlet density}

In this section, 
we 
prove some formulas 
about the Dirichlet density.

\begin{lem}\label{4.7}
Let $L/K$ be a finite extension of number fields and $S$ a set of primes of $K$.
If $S_f \subset \cs(L/K)$,
then 
\begin{equation*}
\delta_{\sup}(S(L)) = [L:K]\delta_{\sup}(S)
\text{ and }\delta_{\inf}(S(L)) = [L:K]\delta_{\inf}(S).
\end{equation*}
\end{lem}
\begin{proof}
Any prime in $\cs(L/K)$ splits completely into $[L:K]$ primes in $L/K$.
Therefore, if $S_f \subset \cs(L/K)$, then 
we obtain 
\begin{equation*}
\delta_{\sup}(S(L)) 
=\limsup_{s \to 1+0} \frac{\sum_{\p \in S_f} [L:K]\frak{N}(\p)^{-s}}{\log{\frac{1}{s-1}}}
=[L:K]\delta_{\sup}(S).
\end{equation*}
The proof of the assertion for $\delta_{\inf}$ is similar. 
\end{proof}

\begin{lem}\label{3.2}
For $i=1,2$, 
let $K_i$ be a number field, $S_i$ a set of primes of $K_i$, 
$T_i \subset S_{i,f}$ 
and $\sigma :G_{K_1,S_1}\isom G_{K_2,S_2}$ 
an 
isomorphism.
Assume that 
the good local correspondence between $T_1$ and $T_2$
holds for $\sigma$.
Then 
$\delta_{\sup}(T_1)=\delta_{\sup}(T_2)$.
Further, 
assume that for $i=1,2$, $K_i/\bQ$ is Galois.
Then 
for $i=1,2$, 
$\delta_{\sup}(T_i(K_1K_2))=[K_1K_2:K_i]\delta_{\sup}(T_i)$.
The similar assertions hold for $\delta_{\inf}$.
\end{lem}
\begin{proof}
Since 
the good local correspondence between $T_1$ and $T_2$
holds for $\sigma$, 
\begin{equation*}
\begin{split} 
\delta_{\sup}(T_1) &=\limsup_{s \to 1+0} \frac{\sum_{\p_1 \in T_1} \frak{N}(\p_1)^{-s}}{\log{\frac{1}{s-1}}} 
=\limsup_{s \to 1+0} \frac{\sum_{\p_2 \in T_2} \frak{N}(\p_2)^{-s}}{\log{\frac{1}{s-1}}}
=\delta_{\sup}(T_2).
\end{split}
\end{equation*}
By symmetry, it suffices to show the second assertion only for $i=1$.
Omitting finite sets from $T_1$ and $T_2$, 
we may assume that 
for $i=1,2$, 
any prime in $T_i(K_1K_2)$
is not ramified over $\bQ$.
We first show that 
$
T_1  \cap \cs(K_1/\bQ)(K_1) = T_1  \cap \cs(K_1K_2/\bQ)(K_1).
$
It is obvious that the right side is contained in the left side, so we show the converse.
Take $\p_1 \in T_1  \cap \cs(K_1/\bQ)(K_1)$ and set $p \defeq \p_1|_{\bQ}$.
Then $\p_1$ is of residual degree $1$ and we have $p \in \cs(K_1/\bQ)$.
Since the good local correspondence between $T_1$ and $T_2$ holds for $\sigma$, $\sigma_{\ast,K_1}(\p_1)$ is also above $p$ and of residual degree $1$.
As $K_2/\bQ$ is Galois, $p \in \cs(K_2/\bQ)$.
Thus, we have $p \in \cs(K_1K_2/\bQ)$, so that $\p_1 \in \cs(K_1K_2/\bQ)(K_1)$.
Therefore, 
\begin{equation*}
\begin{split} 
\delta_{\sup}(T_1(K_1K_2)) 
&= \delta_{\sup}(T_1(K_1K_2) \cap \cs(K_1K_2/\bQ)(K_1K_2))\\
&=\delta_{\sup}((T_1 \cap \cs(K_1K_2/\bQ)(K_1))(K_1K_2))\\
&=[K_1K_2:K_1]\delta_{\sup}(T_1 \cap \cs(K_1K_2/\bQ)(K_1))\\
&=[K_1K_2:K_1]\delta_{\sup}(T_1 \cap \cs(K_1/\bQ)(K_1))\\
&=[K_1K_2:K_1]\delta_{\sup}(T_1),
\end{split}
\end{equation*}
where the first and the fifth equalities follow from the fact that 
$K_1K_2$ and $K_1$ are Galois over $\bQ$ 
and that 
the primes of residual degrees $\geq 2$ do not contribute to the density, 
and the third equality follows from 
Lemma \ref{4.7}.
The proof of the assertions for $\delta_{\inf}$ is similar.
\end{proof}

By this lemma, we obtain another way to show the existence of an isomorphism of fields.

\begin{prop}\label{3.3}
For $i=1,2$, 
let $K_i$ be a number field, $S_i$ a set of primes of $K_i$, 
$T_i \subset S_{i,f}$ 
and $\sigma :G_{K_1,S_1}\isom G_{K_2,S_2}$ 
an 
isomorphism.
Assume that the following conditions hold:
\begin{itemize}
\item[(a)]
$K_i/\bQ$ is Galois for $i=1,2$.
\item[(b)]
The good local correspondence between $T_1$ and $T_2$
holds for $\sigma$.

\item[(c)]
$\delta_{\sup}(T_i) > 1/2$ for 
one $i$.
\end{itemize}
Then 
$K_1 \simeq K_2$.
\end{prop}
\begin{proof}
By Lemma \ref{3.2} 
and conditions (b), (c), 
we have $\delta_{\sup}(T_1)=\delta_{\sup}(T_2) > 1/2$ 
and
$\delta_{\sup}(T_i(K_1K_2))=[K_1K_2:K_i]\delta_{\sup}(T_i)$ for $i=1,2$.
Hence we have $[K_1K_2:K_i] = 1$ for $i=1,2$, 
so that $K_1 = K_2$.
\end{proof}

The results in the rest of this section 
are only used in the proof of Theorem \ref{4.9}.

\begin{definition}
Let $K$ be a number field and $S$ a set of primes of $K$. 
Then we set 
\begin{equation*}
S^{\s} \defeq \{ \p \in S_f \mid 
\text{the local degree of $\p$ is $1$}\}
,
S^{\s,\ff} \defeq \{ \p \in S^{\s} \mid 
\text{$P_{K, \p \vert_\bQ} \subset S^{\s}$}\}
=P_{S^{\s}}(K).\footnote{``s'' is an abbreviation for ``split'' and ``ff'' for ``fiber-full''.}\end{equation*}
\end{definition}

\begin{lem}\label{4.4}
Let $K$ be a number field, $S$ a set of primes of $K$. 
Then 
\begin{equation*}
\begin{split}
\delta_{\sup}(S) \leq 
1-\frac{1}{[\widetilde{K}:\bQ]}+\frac{\delta_{\sup}(S^{\s,\ff})}{[K:\bQ]}
\text{ and }
\delta_{\inf}(S) \leq 
1-\frac{1}{[\widetilde{K}:\bQ]}+\frac{\delta_{\inf}(S^{\s,\ff})}{[K:\bQ]}.
\end{split}
\end{equation*}
\end{lem}
\begin{proof}
Since 
$S_f 
\subset (P_{K,f} \setminus P_K^{\s,\ff}) \coprod (S_f \cap P_K^{\s,\ff})$, 
we have for $s>1$, 
\begin{equation*}
\begin{split}
\frac{\sum_{\p \in S_f} \frak{N}(\p)^{-s}}{\log{\frac{1}{s-1}}}
&\leq 
\frac{\sum_{\p \in P_{K,f} \setminus P_K^{\s,\ff}} \frak{N}(\p)^{-s}}{\log{\frac{1}{s-1}}} 
+ \frac{\sum_{\p \in S_f \cap P_K^{\s,\ff}} \frak{N}(\p)^{-s}}{\log{\frac{1}{s-1}}}.
\end{split}
\end{equation*}
Here, 
$P_K^{\s,\ff} = P_{P_K^{\s}}(K) = 
\cs(K/\bQ)(K) = \cs(\widetilde{K}/\bQ)(K)$ 
and 
prime numbers in $\cs(\widetilde{K}/\bQ)$ 
split completely into $[K:\bQ]$ primes in $K/\bQ$.
Therefore, we have for $s>1$, 
\begin{equation*}
\begin{split}
\frac{\sum_{\p \in P_{K,f} \setminus P_K^{\s,\ff}} \frak{N}(\p)^{-s}}{\log{\frac{1}{s-1}}} 
&=
\frac{\sum_{\p \in P_{K,f}} \frak{N}(\p)^{-s}}{\log{\frac{1}{s-1}}} 
-
\frac{\sum_{\p \in P_K^{\s,\ff}} \frak{N}(\p)^{-s}}{\log{\frac{1}{s-1}}}\\
&=
\frac{\sum_{\p \in P_{K,f}} \frak{N}(\p)^{-s}}{\log{\frac{1}{s-1}}} 
-
[K:\bQ] \cdot 
\frac{\sum_{p \in \cs(\widetilde{K}/\bQ)} p^{-s}}{\log{\frac{1}{s-1}}}.
\end{split}
\end{equation*}
On the other hand, 
$S \cap P_K^{\s,\ff} = S^{\s} \cap P_K^{\s,\ff} 
= S^{\s,\ff} \coprod ((S^{\s} \setminus S^{\s,\ff}) \cap P_K^{\s,\ff})$.
Further, 
\begin{equation*}(S^{\s} \setminus S^{\s,\ff}) \cap P_K^{\s,\ff} \subset P_K^{\s,\ff} \setminus S^{\s,\ff} = (\cs(\widetilde{K}/\bQ) \setminus P_{S^{\s}})(K)\end{equation*} and 
in $K/\bQ$, 
$p \in \cs(\widetilde{K}/\bQ) \setminus P_{S^{\s}}$ 
splits completely into $[K:\bQ]$ primes, 
all of which are not contained in 
$(S^{\s} \setminus S^{\s,\ff}) \cap P_K^{\s,\ff}$.
Hence 
we have for $s>1$, 
\begin{equation*}
\begin{split}
&\frac{\sum_{\p \in S \cap P_K^{\s,\ff}} \frak{N}(\p)^{-s}}{\log{\frac{1}{s-1}}}
\leq
\frac{\sum_{\p \in S^{\s,\ff}} \frak{N}(\p)^{-s}}{\log{\frac{1}{s-1}}} 
+ 
\frac{\sum_{\p \in (S^{\s} \setminus S^{\s,\ff}) \cap P_K^{\s,\ff}} \frak{N}(\p)^{-s}}{\log{\frac{1}{s-1}}}\\
&\leq
\frac{\sum_{\p \in S^{\s,\ff}} \frak{N}(\p)^{-s}}{\log{\frac{1}{s-1}}} 
+ 
\frac{\sum_{p \in \cs(\widetilde{K}/\bQ) \setminus P_{S^{\s}}} ([K:\bQ]-1)p^{-s}}{\log{\frac{1}{s-1}}}\\
&=
\frac{\sum_{\p \in S^{\s,\ff}} \frak{N}(\p)^{-s}}{\log{\frac{1}{s-1}}} 
+ 
\frac{\sum_{p \in \cs(\widetilde{K}/\bQ)} ([K:\bQ]-1)p^{-s}}{\log{\frac{1}{s-1}}}
- 
\frac{\sum_{p \in P_{S^{\s}}} ([K:\bQ]-1)p^{-s}}{\log{\frac{1}{s-1}}}\\
&=
\frac{\sum_{\p \in S^{\s,\ff}} \frak{N}(\p)^{-s}}{\log{\frac{1}{s-1}}} 
+ 
([K:\bQ]-1) \cdot 
\frac{\sum_{p \in \cs(\widetilde{K}/\bQ)} p^{-s}}{\log{\frac{1}{s-1}}}
- 
\frac{[K:\bQ]-1}{[K:\bQ]} \cdot 
\frac{\sum_{\p \in S^{\s,\ff}} \frak{N}(\p)^{-s}}{\log{\frac{1}{s-1}}}\\
&=
\frac{1}{[K:\bQ]} \cdot 
\frac{\sum_{\p \in S^{\s,\ff}} \frak{N}(\p)^{-s}}{\log{\frac{1}{s-1}}} 
+ 
([K:\bQ]-1) \cdot 
\frac{\sum_{p \in \cs(\widetilde{K}/\bQ)} p^{-s}}{\log{\frac{1}{s-1}}}.
\end{split}
\end{equation*}
By the Chebotarev density theorem, 
$\delta(\cs(\widetilde{K}/\bQ))=1/[\widetilde{K}:\bQ]$. 
Thus, we obtain 
\begin{equation*}
\begin{split}
\delta_{\sup}(S) &\leq 
\limsup_{s \to 1+0} \left( 
\frac{\sum_{\p \in P_{K,f}} \frak{N}(\p)^{-s}}{\log{\frac{1}{s-1}}} 
-
\frac{\sum_{p \in \cs(\widetilde{K}/\bQ)} p^{-s}}{\log{\frac{1}{s-1}}}
+
\frac{1}{[K:\bQ]} \cdot 
\frac{\sum_{\p \in S^{\s,\ff}} \frak{N}(\p)^{-s}}{\log{\frac{1}{s-1}}} 
\right) \\
&= 
\delta_{}(P_{K,f})
- \delta_{}(\cs(\widetilde{K}/\bQ))
+
\frac{\delta_{\sup}(S^{\s,\ff})}{[K:\bQ]}
=1-
\frac{1}{[\widetilde{K}:\bQ]}
+
\frac{\delta_{\sup}(S^{\s,\ff})}{[K:\bQ]}.
\end{split}
\end{equation*}
The proof of the assertion for $\delta_{\inf}$ is similar.
\end{proof}

\begin{lem}\label{4.10}
Let $K$ be a number field and $T \subset S$ sets of primes of $K$.
Assume that $S$ has Dirichlet density.
Then 
\begin{equation*}
\delta_{\sup}(S \setminus T) + \delta_{\inf}(T) = \delta(S).
\end{equation*}
\end{lem}
\begin{proof}
Take a sequence $\{ s_n \}_{n \in \bZ_{>0}}$ of real numbers with $s_n > 1$ such that 
$\lim_{n \to \infty} s_n = 1$. 
Then 
$\lim_{n \to \infty} 
\frac{\sum_{\p \in S} \frak{N}(\p)^{-s_n}}{\log{\frac{1}{s_n-1}}} 
= \delta(S)$.
The sequence $\{ \frac{\sum_{\p \in T} \frak{N}(\p)^{-s_n}}{\log{\frac{1}{s_n-1}}}  \}_{n \in \bZ_{>0}}$ is convergent 
if and only if 
the sequence $\{ \frac{\sum_{\p \in S \setminus T} \frak{N}(\p)^{-s_n}}{\log{\frac{1}{s_n-1}}}  \}_{n \in \bZ_{>0}}$ is convergent.
Hence, if one of them is convergent, 
both of them are convergent and 
we have 
\begin{equation*}
\begin{split}
\lim_{n \to \infty} 
\frac{\sum_{\p \in S \setminus T} \frak{N}(\p)^{-s_n}}{\log{\frac{1}{s_n-1}}} 
+
\lim_{n \to \infty} 
\frac{\sum_{\p \in T} \frak{N}(\p)^{-s_n}}{\log{\frac{1}{s_n-1}}} 
=
\delta(S).
\end{split}
\end{equation*}
Therefore, 
$\{ \frac{\sum_{\p \in T} \frak{N}(\p)^{-s_n}}{\log{\frac{1}{s_n-1}}}  \}_{n \in \bZ_{>0}}$ is convergent to $\delta_{\inf}(T)$ 
if and only if 
$\{ \frac{\sum_{\p \in S \setminus T} \frak{N}(\p)^{-s_n}}{\log{\frac{1}{s_n-1}}}  \}_{n \in \bZ_{>0}}$ is convergent to $\delta_{\sup}(S \setminus T)$.
Thus, we obtain the assertion.
\end{proof}

\section{Main results}

In this section, we prove the main theorems 
in this paper 
using the results obtained so far.

\begin{theorem}\label{4.0}
For $i=1,2$, 
let $K_i$ be a number field, $S_i$ a set of primes of $K_i$ 
with $P_{K_i,\infty} \subset S_i$ 
and $\sigma :G_{K_1,S_1}\isom G_{K_2,S_2}$ 
an 
isomorphism.
Assume that the following conditions hold:
\begin{itemize}

\item[(a)]
$\# P_{S_i,f} \geq 2$ for $i=1,2$.

\item[(b)]
For one $i$, 
$\delta(P_{S_i,f} \cap \cs(K_i/\bQ)) \neq 0$.

\item[(c)]
For the $i$ in condition (b), 
there exists 
a prime number $l \in P_{S_1,f} \cap P_{S_2,f}$ 
such that 
$S_{3-i}$ 
satisfies condition $(\star_l)$.
\end{itemize}
Then 
$K_1 \simeq K_2$.
\end{theorem}
\begin{proof}
By symmetry, we may assume that the $i$ in condition (b) is $1$.
By Lemma \ref{4.7}, 
we have 
$\delta_{\sup}(S_1) \geq \delta_{\sup}((P_{S_1,f} \cap \cs(K_1/\bQ))(K_1)) 
= [K_1 : \bQ]\delta_{\sup}(P_{S_1,f} \cap \cs(K_1/\bQ)) 
> 0$.
Hence, by Proposition \ref{1.20}, 
$S_1$ satisfies condition $(\star_l)$ 
for the prime number $l$ in condition (c).
Therefore, 
the local correspondence between $S_{1,f}$ and $S_{2,f}$
holds for $\sigma$ by Theorem \ref{2.5}, 
so that 
$P_{S_1,f} = P_{S_2,f}$ 
by Proposition \ref{2.7}.

Next, we define subextensions $L_1$, $L_1'$ and $M_1$ of $K_{1,{S_1}}/K_1$ as follows. 
Take subsets $S_0, S_0' \subset P_{S_1,f}$ 
with $\# S_0 = \infty$, $\# S_0' = \infty$ and $S_0 \cap S_0' = \emptyset$ 
such that 
$\widetilde{K_1}$, $\bQ(\mu_{\ltilde} ; l \in S_0)$ and $\bQ(\mu_{\ltilde} ; l \in S_0')$ are linearly disjoint. 
By Proposition \ref{4f.6}, 
there exists a finite subextension $L_0$ (resp. $L_0'$) of $\bQ(\mu_{\ltilde} ; l \in S_0)/\bQ$ (resp. $\bQ(\mu_{\ltilde} ; l \in S_0')/\bQ$) such that 
$L_0$ (resp. $L_0'$) is totally imaginary and 
$\delta_{\sup}(P_{S_1,f} \cap \cs(K_1/\bQ) \cap \cs(L_0/\bQ)) > 0$ (resp. $\delta_{\sup}(P_{S_1,f} \cap \cs(K_1/\bQ) \cap \cs(L_0/\bQ) \cap \cs(L_0'/\bQ)) > 0$).
Then we set $L_1 \defeq K_1L_0$, $L_1' \defeq K_1L_0'$ and $M_1 \defeq L_1L_1' = K_1L_0L_0'$.

We write subextensions $L_2$, $L_2'$ and $M_2$ of $K_{2,{S_2}}/K_2$ for the finite subextensions corresponding to 
$\sigma(G_{L_1,S_1(L_1)})$, $\sigma(G_{L_1',S_1(L_1')})$ and $\sigma(G_{M_1,S_1(M_1)})$, respectively.
Note that 
$L_1$ and $L_1'$ are totally imaginary, and 
$L_i \cap L_i' = K_i$ for $i=1,2$.
By Remark \ref{2.6} and Proposition \ref{2.7}, 
we have 
$P_{S_1,f} \cap \cs(M_1/\bQ) =P_{S_2,f} \cap \cs(M_2/\bQ)$ and 
the good local correspondence between $P_{S_1,f}(M_1)$ and $P_{S_2,f}(M_2)$ holds for $\sigma|_{G_{M_1,S_1(M_1)}} : G_{M_1,S_1(M_1)}\isom G_{M_2,S_2(M_2)}$.
Now, we set 
$M \defeq \widetilde{M_1}\widetilde{M_2}$.
Then $M/\bQ$ is Galois and we have 
\begin{equation*}
\begin{split}
P_{S_1,f} \cap \cs(M/\bQ) &= P_{S_1,f} \cap \cs(\widetilde{M_1}/\bQ) \cap \cs(\widetilde{M_2}/\bQ)\\
&= P_{S_1,f}\cap \cs(M_1/\bQ) \cap \cs(M_2/\bQ)\\
&= P_{S_1,f} \cap \cs(M_1/\bQ)\\
&= P_{S_1,f} \cap \cs(K_1/\bQ)\cap \cs(L_0/\bQ)\cap \cs(L_0'/\bQ),
\end{split}
\end{equation*}
so that $\delta_{\sup}(P_{S_1,f} \cap \cs(M/\bQ)) > 0$.
By Lemma \ref{4.7}, 
\begin{equation*}\delta_{\sup}(P_{S_1,f}(M)) 
= [M : \bQ]\delta_{\sup}(P_{S_1,f} \cap \cs(M/\bQ)) 
> 0.\end{equation*}
Take $l \in P_{S_1,f} (= P_{S_2,f})$.
By Proposition \ref{4f.3}, 
there exists $\tau \in G(M/\bQ)$ such that 
\begin{equation*}\overline{\sigma}_{M_1} \circ \pi_{M/M_1, l} = (\tau|_{M_2})^\ast \circ \pi_{M/\tau(M_2), l},\end{equation*} 
where $\overline{\sigma}_{M_1}:\Gamma_{M_1, l} \isom \Gamma_{M_2, l}$ is the isomorphism induced by $\sigma|_{G_{M_1,S_1(M_1)}} : G_{M_1,S_1(M_1)}\isom G_{M_2,S_2(M_2)}$. 
Then 
\begin{equation*}
\begin{split}
\overline{\sigma}_{L_1} \circ \pi_{M/L_1, l} 
&=\overline{\sigma}_{L_1} \circ \pi_{M_1/L_1, l} \circ \pi_{M/M_1, l} \\
&=\pi_{M_2/L_2, l} \circ \overline{\sigma}_{M_1} \circ \pi_{M/M_1, l} \\
&=\pi_{M_2/L_2, l} \circ (\tau|_{M_2})^\ast \circ \pi_{M/\tau(M_2), l} \\
&=(\tau|_{L_2})^\ast \circ \pi_{\tau(M_2)/\tau(L_2), l} \circ \pi_{M/\tau(M_2), l} \\
&= (\tau|_{L_2})^\ast \circ \pi_{M/\tau(L_2), l}, 
\end{split}
\end{equation*}
where $\overline{\sigma}_{L_1}:\Gamma_{L_1, l} \isom \Gamma_{L_2, l}$ is the isomorphism induced by $\sigma|_{G_{L_1,S_1(L_1)}} : G_{L_1,S_1(L_1)}\isom G_{L_2,S_2(L_2)}$. 
Therefore, ${(\tau|_{L_2})^\ast}^{-1} \circ \overline{\sigma}_{L_1} \circ \pi_{M/L_1, l} = \pi_{M/\tau(L_2), l}$, 
so that we have 
$L_1^{(\infty,l)}\tau(L_2) = L_1\tau(L_2)^{(\infty,l)}$ 
by Lemma \ref{4f.4}.
Then $L_1^{(\infty,l)}M = M\tau(L_2)^{(\infty,l)}$.
By Proposition \ref{4f.5}, we obtain $L_1 = \tau(L_2)$.
Similarly, we have $L_1' = \tau(L_2')$.
Thus, 
$K_1 = L_1 \cap L_1' = \tau(L_2 \cap L_2') = \tau(K_2)$.
\end{proof}

\begin{cor}\label{4.0.5}
For $i=1,2$, 
let $K_i$ be a number field, $S_i$ a set of primes of $K_i$ 
with $P_{K_i,\infty} \subset S_i$ 
and $\sigma :G_{K_1,S_1}\isom G_{K_2,S_2}$ 
an 
isomorphism.
Assume that the following conditions hold:
\begin{itemize}

\item[(a)]
$P_{S_1,f} \cap P_{S_2,f} \neq \emptyset$.

\item[(b)]
$\delta(P_{S_i,f} \cap \cs(K_i/\bQ)) \neq 0$ for $i=1,2$.

\end{itemize}
Then 
$K_1 \simeq K_2$.
\end{cor}
\begin{proof}
By the assumption, 
we can show easily that the conditions in Theorem \ref{4.0} hold.
\end{proof}

\begin{lem}\label{4.0.6}
Let $K$ be a number field, and $S$ a set of primes of $K$ with $P_{\infty} \subset S$.
If $S$ is finite, then we have 
\begin{equation*}\# S - \sum_{\p \in P_2 \setminus S} ( [K_\p:\bQ_{2}] + 1 ) \leq \dim_{\bF_2}(G_{K,S}^{\ab}/(G_{K,S}^{\ab})^{2}) < \infty.\end{equation*}
In particular, $\# S = \infty$ if and only if $\dim_{\bF_2}(G_{K,S}^{\ab}/(G_{K,S}^{\ab})^{2}) = \infty$.
\end{lem}
\begin{proof}
Assume $S$ is finite.
Then, by \cite{NSW}, (8.3.20) Theorem, 
$H^1(G_{K,S},\bF_2) \simeq \Hom(G_{K,S}^{\ab}/(G_{K,S}^{\ab})^{2}, \bF_2)$ is finite.
Hence 
$\dim_{\bF_2}(G_{K,S}^{\ab}/(G_{K,S}^{\ab})^{2})$ is finite.
Set $L \defeq K( \sqrt{\mathcal O_{K,S}^{\times}} )$.
By \cite{Neukirch3}, Chapter V, (3.3) Lemma, 
$L/K$ is unramified outside $P_2 \cup S$.
Since the canonical homomorphism: $\mathcal O_{K,S}^{\times}/(\mathcal O_{K,S}^{\times})^2 \to K^{\times}/K^{\times 2}$ is injective, 
we have the isomorphism: 
$G(L/K) \simeq \Hom(O_{K,S}^{\times}/(\mathcal O_{K,S}^{\times})^2, \mu_2)$ by Kummer theory.
By Dirichlet's $S$-unit theorem and $\mu_2 \subset \mathcal O_{K,S}^{\times}$, 
we have 
$\dim_{\bF_2}(O_{K,S}^{\times}/(\mathcal O_{K,S}^{\times})^2) = \# S
$, 
so that $\dim_{\bF_2}G(L/K) = \# S$.
For $\p \in P_2 \setminus S$, 
we write $I_\p$ for the inertia subgroup of $D_{\p,L/K} \subset G(L/K)$.
Then, by local class field theory, we have 
$\dim_{\bF_2} I_\p \leq [K_\p:\bQ_{2}] + 1$, so that 
\begin{equation*}
\begin{split}
\dim_{\bF_2}(G_{K,S}^{\ab}/(G_{K,S}^{\ab})^{2}) 
&\geq \dim_{\bF_2} (G(L/K)/\langle I_\p \mid \p \in P_2 \setminus S \rangle) \\
&\geq \dim_{\bF_2}G(L/K) - \sum_{\p \in P_2 \setminus S} \dim_{\bF_2} I_\p\\
&\geq \# S - \sum_{\p \in P_2 \setminus S} ( [K_\p:\bQ_{2}] + 1 ).
\end{split}
\end{equation*}
The last assertion follows from the inequality.
\end{proof}

\begin{rem}
In Theorem \ref{4.0}, 
if, for one $i$, $K_i$ is totally real and 
there exists a prime number $l \in P_{S_1,f} \cap P_{S_2,f}$ 
such that 
the Leopoldt conjecture is true for the pair ($K_i$, $l$), 
then we can omit 
condition (c).
Indeed, 
then, 
for $i=1,2$, 
$K_i$ has only one $\bZ_l$-extension, 
so that 
condition $(\star_l)$ for $S_i$ is equivalent to the condition $\# S_i = \infty$ 
by Remark \ref{1.18}.
Therefore, 
by condition (b) and Lemma \ref{4.0.6}, 
condition (c) holds.
\end{rem}

\begin{theorem}\label{4.1}
For $i=1,2$, 
let $K_i$ be a number field, $S_i$ a set of primes of $K_i$ 
with $P_{K_i,\infty} \subset S_i$ 
and $\sigma :G_{K_1,S_1}\isom G_{K_2,S_2}$ 
an 
isomorphism.
Assume that the following conditions hold:
\begin{itemize}

\item[(a)]
$\# P_{S_i,f} \geq 2$ for $i=1,2$.

\item[(b)]
For one $i$, there exist a totally real subfield $K_{i,0} \subset K_i$ and a set of nonarchimedean primes $T_{i,0}$ of $K_{i,0}$ such that $T_{i,0}(K_i) \subset S_{i,f}$ and $\delta(T_{i,0}(\widetilde{K_1}\widetilde{K_2})) \neq 0$.

\item[(c)]
For the $i$ in condition (b), 
there exists 
a prime number $l \in P_{S_1,f} \cap P_{S_2,f}$ 
such that 
$S_{3-i}$ satisfies condition $(\star_l)$.

\item[(d)]
$K_i$ has a complex prime for one $i$.

\end{itemize}
Then 
$K_1 \simeq K_2$.
\end{theorem}
\begin{proof}
By symmetry, we may assume that the $i$ in condition (b) is $1$.
We set $T_1 \defeq T_{1,0}(K_1)$.
By Lemma \ref{4.7}, 
we have 
$\delta_{\sup}(S_1) 
\geq \delta_{\sup}(T_1 \cap \cs(\widetilde{K_1}\widetilde{K_2}/K_1)) 
= \delta_{\sup}(T_1(\widetilde{K_1}\widetilde{K_2})) / [ \widetilde{K_1}\widetilde{K_2} : K_1 ]
> 0$.
Hence, by Proposition \ref{1.20}, 
$S_1$ satisfies condition $(\star_l)$ 
for the prime number $l$ in condition (c).
Therefore, 
the local correspondence 
(resp. the good local correspondence) 
between $S_{1,f}$ (resp. $T_1$) and $S_{2,f}$ (resp. $T_2$) 
holds for $\sigma$ by Theorem \ref{2.5}, 
where 
$T_2 \subset S_{2,f}$ is the set corresponding to $T_1$ under the local correspondence between $S_{1,f}$ and $S_{2,f}$.
Then, we have $P_{S_1,f} = P_{S_2,f}$ by Proposition \ref{2.7}.

Take $l \in P_{S_1,f}$.
By Proposition \ref{4f.3}, 
there exists $\tau \in G(\widetilde{K_1}\widetilde{K_2}/\bQ)$ such that 
\begin{equation*}\overline{\sigma}_{K_1} \circ \pi_{\widetilde{K_1}\widetilde{K_2}/K_1, l} = (\tau|_{K_2})^\ast \circ \pi_{\widetilde{K_1}\widetilde{K_2}/\tau(K_2), l},\end{equation*}
where $\overline{\sigma}_{K_1}:\Gamma_{K_1, l} \isom \Gamma_{K_2, l}$ is the isomorphism induced by $\sigma$. 
Therefore, 
\begin{equation*}{(\tau|_{K_2})^\ast}^{-1} \circ \overline{\sigma}_{K_1} \circ \pi_{\widetilde{K_1}\widetilde{K_2}/K_1, l} = \pi_{\widetilde{K_1}\widetilde{K_2}/\tau(K_2), l},\end{equation*}
so that 
we have 
$K_1^{(\infty,l)}\tau(K_2) = K_1\tau(K_2)^{(\infty,l)}$ 
by Lemma \ref{4f.4}.
Then $K_1^{(\infty,l)}\widetilde{K_1}\widetilde{K_2} = \widetilde{K_1}\widetilde{K_2}\tau(K_2)^{(\infty,l)}$.
By Proposition \ref{4f.5}, we obtain $K_1 = \tau(K_2)$.
\end{proof}

\begin{cor}\label{4.2}
For $i=1,2$, 
let $K_i$ be a number field, $S_i$ a set of primes of $K_i$ 
with $P_{K_i,\infty} \subset S_i$ 
and $\sigma :G_{K_1,S_1}\isom G_{K_2,S_2}$ 
an 
isomorphism.
Assume that the following conditions hold:
\begin{itemize}

\item[(a)]
$\# P_{S_i,f} \geq 2$ for $i=1,2$.

\item[(b)]
For one $i$, there exist a totally real subfield $K_{i,0} \subset K_i$ and a set of nonarchimedean primes $T_{i,0}$ of $K_{i,0}$ such that $T_{i,0}(K_i) \subset S_{i,f}$ and $\delta(T_{i,0}(K_i)) \neq 0$.

\item[(c)]
For the $i$ in condition (b), 
there exists 
a prime number $l \in P_{S_1,f} \cap P_{S_2,f}$ 
such that 
$S_{3-i}$ satisfies condition $(\star_l)$.

\item[(d)]
$K_i/\bQ$ is Galois for $i=1,2$ and $K_i$ is totally imaginary for one $i$.

\end{itemize}
Then 
$K_1 \simeq K_2$.
\end{cor}
\begin{proof}
By symmetry, we may assume that the $i$ in condition (b) is $1$.
We set $T_1 \defeq T_{1,0}(K_1)$.
Then $\delta_{\sup}(S_1) \geq \delta_{\sup}(T_1) > 0$.
As in the first paragraph of the proof of Theorem \ref{4.1}, 
the local correspondence 
(resp. the good local correspondence) 
between $S_{1,f}$ (resp. $T_1$) and $S_{2,f}$ (resp. $T_2$) 
holds for $\sigma$, 
where 
$T_2 \subset S_{2,f}$ is the set corresponding to $T_1$ under the local correspondence between $S_{1,f}$ and $S_{2,f}$.
Then, by Lemma \ref{3.2}, we have 
$\delta_{\sup}(T_1(K_1K_2))=[K_1K_2:K_1]\delta_{\sup}(T_1) >0$.
Therefore, by Theorem \ref{4.1}, we obtain $K_1 \simeq K_2$.
\end{proof}


\begin{theorem}\label{4.3}
For $i=1,2$, 
let $K_i$ be a number field, $S_i$ a set of primes of $K_i$ 
with $P_{K_i,\infty} \subset S_i$ 
and $\sigma :G_{K_1,S_1}\isom G_{K_2,S_2}$ 
an isomorphism.
Assume that the following conditions hold:
\begin{itemize}
\item[(a)]
$\# P_{S_i,f} \geq 2$ for $i=1,2$.

\item[(b)]
For one $i$, there exist a totally real subfield $K_{i,0} \subset K_i$ and a set of nonarchimedean primes $T_{i,0}$ of $K_{i,0}$ such that $T_{i,0}(K_i) \subset S_{i,f}$ and 
$\delta_{\sup}(T_{i,0}(K_i)) > 1/2$.

\item[(c)]
For the $i$ in condition (b), 
there exists 
a prime number $l \in P_{S_1,f} \cap P_{S_2,f}$ 
such that 
$S_{3-i}$ satisfies condition $(\star_l)$.

\item[(d)]
$K_i/\bQ$ is Galois for $i=1,2$.

\end{itemize}
Then 
$K_1 \simeq K_2$.
\end{theorem}
\begin{proof}
It is enough to modify the proof of Corollary \ref{4.2} 
by using Proposition \ref{3.3} instead of Theorem \ref{4.1}.
\end{proof}

\begin{theorem}\label{4.9}
For $i=1,2$, 
let $K_i$ be a number field, $S_i$ a set of primes of $K_i$ 
with $P_{K_i,\infty} \subset S_i$ 
and $\sigma :G_{K_1,S_1}\isom G_{K_2,S_2}$ 
an 
isomorphism.
Assume that 
$\delta_{\sup}(S_1) + \delta_{\inf}(S_2) > 2-${\large$\frac{1}{[\widetilde{K_1}\widetilde{K_2}:\bQ]}$}.
Then 
$K_1 \simeq K_2$.
\end{theorem}
\begin{proof}
We show that the conditions in Theorem \ref{4.0} hold.
By Lemma \ref{4.7} and Lemma \ref{4.4}, 
\begin{equation*}
\begin{split}
\delta_{\sup}(P_{S_1^{\s}}) + \delta_{\inf}(P_{S_2^{\s}})
&= \frac{\delta_{\sup}(S_1^{\s,\ff})}{[K_1:\bQ]}
+ \frac{\delta_{\inf}(S_2^{\s,\ff})}{[K_2:\bQ]}
\geq 
\delta_{\sup}(S_1) + \delta_{\inf}(S_2) - 2 
+ \frac{1}{[\widetilde{K_1}:\bQ]} + \frac{1}{[\widetilde{K_2}:\bQ]}\\
&> 
\frac{1}{[\widetilde{K_1}:\bQ]} + \frac{1}{[\widetilde{K_2}:\bQ]}
-
\frac{1}{[\widetilde{K_1}\widetilde{K_2}:\bQ]}.
\qquad\qquad\qquad\qquad\qquad\qquad\quad(5.1)
\end{split}
\end{equation*}
Note that 
$P_{S_1^{\s}} \subset \cs(\widetilde{K_1}/\bQ)$ and $P_{S_2^{\s}} \subset \cs(\widetilde{K_2}/\bQ)$.
By the Chebotarev density theorem, 
$\delta(\cs(\widetilde{K_1}/\bQ))=1/[\widetilde{K_1}:\bQ]$, 
$\delta(\cs(\widetilde{K_2}/\bQ))=1/[\widetilde{K_2}:\bQ]$ 
and 
\begin{equation*}\delta(\cs(\widetilde{K_1}/\bQ) \cap \cs(\widetilde{K_2}/\bQ)) 
= \delta(\cs(\widetilde{K_1}\widetilde{K_2}/\bQ))
= \frac{1}{[\widetilde{K_1}\widetilde{K_2}:\bQ]}.\end{equation*}
Hence 
\begin{equation*}\delta_{\sup}(P_{S_1^{\s}}) \leq 1/[\widetilde{K_1}:\bQ] \text{, } \delta_{\inf}(P_{S_2^{\s}}) \leq 1/[\widetilde{K_2}:\bQ]
\eqno(5.2)\end{equation*} 
and 
by Lemma \ref{4.10}, 
\begin{equation*}\delta_{\sup}((\cs(\widetilde{K_1}/\bQ) \cup \cs(\widetilde{K_2}/\bQ)) \setminus P_{S_2^{\s}}) 
=
\frac{1}{[\widetilde{K_1}:\bQ]} + \frac{1}{[\widetilde{K_2}:\bQ]}
-\frac{1}{[\widetilde{K_1}\widetilde{K_2}:\bQ]} - \delta_{\inf}(P_{S_2^{\s}}).\end{equation*}
Therefore, 
$P_{S_1^{\s}}$ is 
not contained in $(\cs(\widetilde{K_1}/\bQ) \cup \cs(\widetilde{K_2}/\bQ)) \setminus P_{S_2^{\s}}$, so that 
$P_{S_1^{\s}} \cap P_{S_2^{\s}} \neq \emptyset$.
Since $P_{S_i^{\s}} = P_{S_i,f} \cap \cs(K_i/\bQ)$ for $i=1,2$, we have 
$P_{S_1,f} \cap P_{S_2,f} \neq \emptyset$.
By (5.1) and 
(5.2), 
we have 
$\delta_{\sup}(P_{S_1^{\s}})>0$
and $\delta_{\inf}(P_{S_2^{\s}})>0$, 
so that 
condition (b) for $i=1$ holds.
In particular, $\# P_{S_1,f} = \infty$ and $\# P_{S_2,f} = \infty$, 
and hence condition (a) holds.
Since 
\begin{equation*}\delta_{\sup}(S_2) \geq \delta_{\inf}(S_2^{\s,\ff}) 
= [K_2:\bQ]\delta_{\inf}(P_{S_2^{\s}}) 
>0\end{equation*} by Lemma \ref{4.7}, 
condition (c) 
for $i=1$ and 
for any $l \in P_{S_1,f} \cap P_{S_2,f}$
holds by Proposition \ref{1.20}.
\end{proof}

\begin{rem}
If the Leopoldt conjecture is true for all pairs ($K_i$, $p$) where $i=1,2$ and $p$ runs through all prime numbers, 
we can replace the assumption in Theorem \ref{4.9} by the weaker assumption: ``$\delta_{\sup}(S_i) > 
1-1/[\widetilde{K_i}:\bQ]$ for $i=1,2$''.
Indeed, 
by Lemma \ref{4.7} and Lemma \ref{4.4}, 
$\delta_{\sup}(P_{S_i^{\s}})>0$ for $i=1,2$, 
so that 
$P_{S_i,f} \neq \emptyset$.
Hence, by \cite{Ivanov}, Proposition 4.1 (which assumes the validity of the Leopoldt conjecture), 
we have 
$P_{S_1,f} = P_{S_2,f}$.\footnote{
In the assertions of \cite{Ivanov}, Proposition 4.1, 
$S$ is assumed to be finite.
However, even if $S$ is not finite, 
it is easy to modify the proof of the assertion we use, 
by defining $\rk_{\bZ_p}M = \dim_{\bQ_p}((M/\overline{M_{\bZ_p\text{-}\tor}})\otimes_{\bZ_p}\bQ_p)$ for a profinite $\bZ_p$-module $M$, where $\overline{M_{\bZ_p\text{-}\tor}}$ is the closure in $M$ of $M_{\bZ_p\text{-}\tor}$.
}
Therefore, condition (c) in Theorem \ref{4.0} 
for $i=1$ and 
for any $l \in P_{S_1,f}$
holds 
by Proposition \ref{1.20}.
The rest of the proof is the same as that of Theorem \ref{4.9}.
\end{rem}

\noindent
Ryoji Shimizu\\
Research Institute for Mathematical Sciences\\
Kyoto University\\
KYOTO 606-8502\\
Japan\\
shimizur@kurims.kyoto-u.ac.jp\\


\begin{thebibliography}{Y}

\bibitem[Chenevier-Clozel]{Chenevier-Clozel} Chenevier, G., Clozel, L., Corps de nombres peu ramifi\'{e}s et formes automorphes autoduales, J. of the AMS, vol. 22 (2009), no. 2, 467-519.

\bibitem[Ivanov]{Ivanov2} Ivanov, A., Arithmetic and anabelian theorems for stable sets in number fields, Dissertation, Universit\"at Heidelberg, 2013.

\bibitem[Ivanov2]{Ivanov} Ivanov, A., On some anabelian properties of arithmetic curves, Manuscripta Mathematica 144 (2014), no. 3, 545-564.

\bibitem[Ivanov3]{Ivanov3} Ivanov, A., On a generalization of the Neukirch-Uchida theorem, Moscow Mathematical J. 17 (2017), no. 3, 371-383.

\bibitem[Neukirch]{Neukirch} Neukirch, J., Kennzeichnung der $p$-adischen und der endlichen algebraischen Zahlkörper, Invent. Math. 6 (1969), 296–314.


\bibitem[Neukirch3]{Neukirch3} Neukirch, J., Algebraic Number Theory, Grundlehren der Mathematischen Wissenschaften, 322. Springer-Verlag, Berlin, 1999.

\bibitem[NSW]{NSW}
Neukirch, J., Schmidt, A. and Wingberg, K., 
Cohomology of number fields, Second edition, Grundlehren der Mathematischen Wissenschaften, 323. Springer-Verlag, Berlin, 2008.

\bibitem[\text{Sa\"\i di-Tamagawa}]{Saidi-Tamagawa}Sa\"\i di, M. and Tamagawa, A., The $m$-step solvable anabelian geometry of number fields,  preprint,  arXiv:1909.08829.

\bibitem[Serre]{Serre} Serre, J.-P., Abelian $l$-adic representations and elliptic curves, Second edition, Advanced Book Classics, Addison-Wesley, Redwood City, 1989.

\bibitem[Uchida]{Uchida} Uchida, K., Isomorphisms of Galois groups, J. Math. Soc. Japan, 28 (4) (1976), 617–620.

\end{thebibliography}
\end{document}